\documentclass[11pt,fullpage,doublespace]{article}
\usepackage{graphicx} % a package for including diagrams
\usepackage{setspace} % provides the doublespace command below
\usepackage{latexsym} % must have basic LaTeX symbols
\usepackage{amsfonts} % another must have AMS fonts
\usepackage{amsmath} % AMS math stuff like matrices and such
\usepackage{amssymb}
\usepackage{accents}
\usepackage{textcomp}
\usepackage{enumerate}
\usepackage{mathrsfs}
\usepackage{bm}
\usepackage{stmaryrd}
\usepackage[margin=1in]{geometry}
\usepackage{amsthm}
\usepackage{amssymb,latexsym,amsmath}
\usepackage{graphics}
\usepackage{titlesec}
\usepackage{changepage}
\usepackage{lipsum}
\usepackage{thmtools}
\usepackage{comment}
\usepackage[colorlinks=true,linkcolor=blue]{hyperref}
\usepackage{centernot}
\usepackage{mathtools}
\usepackage{tikz}
\usepackage{scrextend}
\usetikzlibrary{matrix,arrows,positioning,scopes}
\declaretheorem[style=definition,qed=$\dashv$,numberwithin=section]
{definition}

\declaretheorem[style=plain,sibling=definition]{theorem}
\declaretheorem[style=plain,sibling=definition]{lemma}

\declaretheorem[style=plain,sibling=definition]{question}
\declaretheorem[style=plain,sibling=definition]{conjecture}

\declaretheorem[style=plain,sibling=definition]{corollary}
\declaretheorem[style=definition,sibling=definition]{remark}
\declaretheorem[style=plain,sibling=definition]{claim}
\declaretheorem[style=plain,sibling=definition]{claim*}

\titleformat{\section}{\normalsize\centering}{\thesection.}{1em}{}
\titleformat{\subsection}{\normalsize\centering}{\thesubsection.}{1em}{}
\titleformat{\subsubsection}{\normalsize}{\thesubsubsection.}{1em}{}
\numberwithin{equation}{section}

\newcommand{\rest}{\restriction}
\newcommand{\la}{\langle}
\newcommand{\ra}{\rangle}

\newcommand{\cp}{{\rm cp }}

\newcommand{\lh}{{\rm lh}}

\newcommand{\crt}{{\rm crt}}

\newcommand{\Ult}{{\rm Ult}}

\def\k{\kappa}
\def\a{\alpha}

\def\d{\delta}

\newcommand{\forces}{\Vdash}

\renewcommand{\models}{\vDash}
\newcommand{\powerset}{{\wp }}

\def\P{{\mathcal{P} }}
\def\W{{\mathcal{W} }}
\def\Q{{\mathcal{ Q}}}

\def\R{{\mathcal R}}

\def\M{{\mathcal{M}}}

\def\T {{\mathcal{T}}}
\def\U{{\mathcal{U}}}
\def\S{{\mathcal{S}}}

\newcommand{\rlem}[1]{Lemma~\ref{#1}}

\newcommand{\lex}{\mathrm{lex}}

\newcommand{\Coll}{\mathrm{Col}}

 \input xy
 \xyoption{all}
\title{Partial Tower Sealing}
\author{Grigor Sargsyan \footnote{Institute of Mathematics of Polish Academy of Sciences, Warsaw, Poland. Email: gsargsyan@impan.pl}\thanks{The first author's work is funded by the National Science Center, Poland under the Weave-Unisono Call, registration number UMO-2023/05/Y/ST1/00194.} \\ Nam Trang \footnote{University of North Texas, Denton, TX, USA. Email: nam.trang@unt.edu} \thanks{The second author is supported by NSF CAREER grant DMS-1945592.}}

\begin{document}
\maketitle

\begin{abstract}
The main result of this paper shows that a weak form of Tower Sealing holds in a generic extension of hod mice with a strong cardinal and a proper class of Woodin cardinals. We show Tower Sealing fails in such extensions in general. We show that this weak form of Tower Sealing (called \textit{Partial Tower Sealing}) implies Sealing and that its consistency strength is below that of $\sf{ZFC} + $``there is a Woodin limit of Woodin cardinals".
\end{abstract}

%To do: Perhaps Weaken the hypo of Theorem \ref{thm:weakTSimpliesSealing}, Add proof showing Tower Sealing fails in minimal hod pair satisfying the hypo of Theorem \ref{thm:weakTS} and in excellent hod mice (does Partial Tower Sealing hod in excellent mice?), add the def of Hom^*, $\kappa$-homgeneously Suslin sets somewhere, Clean up notations, write introduction

\section{Introduction}

This paper formulates a weak form of Woodin's \textit{Tower Sealing, Partial Tower Sealing} cf. Definition \ref{def:WTS}, and shows that this form of Tower Sealing implies \textit{Sealing} under various circumstances. The main result of this paper is Theorem \ref{thm:weakTS}, which shows that Partial Tower Sealing can hold in hod mice; as a result, Partial Tower Sealing is consistent relative to the theory $\sf{ZFC} + $``there is a Woodin limit of Woodin cardinals" ($\sf{WLW}$).

Suppose $g$ is a $V$-generic filter. Let $\Gamma_g^\infty$ be the class of all universally Baire sets in $V[g]$. When $V[g] = V$, we simply write $\Gamma^\infty$. For a cardinal $\kappa$, we write $\mathbb{Q}_{<\kappa}$ for the countable tower forcing as defined in \cite[Definition 2.7.1]{StationaryTower}. 

\begin{definition}\label{def:WTS}
Suppose there is a proper class of Woodin cardinals. Let $\delta$ be a Woodin cardinal. We say that \textit{Partial Tower Sealing} holds at $\delta$ if whenever $g$ is $<\delta$-generic over $V$  and $G\subseteq \mathbb{Q}_{<\delta}$ is $V[g]$-generic, letting $j_G: V[g]\rightarrow M \subseteq V[g][G]$ be the associated generic embedding, then 
\begin{enumerate}
\item $L(\Gamma_g^\infty)\cap \powerset(\mathbb{R}) = \Gamma_g^\infty$.
\item $(\Gamma^\infty_g)^\sharp, (\Gamma^\infty_{g*G})^\sharp$ exist and there is an elementary embedding $l: L(\Gamma_{g*G}^\infty)\rightarrow L(j_G(\Gamma_g^\infty))$ such that $l\rest \Gamma_{g*G}^\infty = \textrm{id}$ and $l$ is an order-preserving surjection from the class of indiscernibles of $L(\Gamma_{g*G}^\infty)$ to the class of indiscernibles of $L(j_G(\Gamma_g^\infty))$.
\end{enumerate}
If Partial Tower Sealing holds at $\delta$ and additionally, $\Gamma_{g*G} = j_G(\Gamma_g^\infty)$, then we say \textit{Tower Sealing} holds at $\delta$.
\end{definition}

\begin{definition}\label{dfn:ub_sealing} Suppose there is a proper class of Woodin cardinals. We say \textit{Sealing} holds at a Woodin cardinal $\delta$ if the following statements hold.
\begin{enumerate}
\item For every $< \delta$ generic $g$ over $V$,  $\powerset(\mathbb{R}_g)\cap L(\Gamma^\infty_g, \mathbb{R}_g)=\Gamma^\infty_g$.
\item  For every $<\delta$ generic $g$ over $V$, for every $<\delta$ generic $h$ over $V[g]$, there is an elementary embedding 
\begin{center}
$j: L(\Gamma^\infty_g, \mathbb{R}_g)\rightarrow L(\Gamma^\infty_{g*h}, \mathbb{R}_h)$.
\end{center}
\end{enumerate}
 such that for every $A\in \Gamma^\infty_g$, $j(A)=A^h.$  
 \end{definition}

We note that if clause (1) of both Definitions \ref{dfn:ub_sealing} and \ref{def:WTS} holds, then $$L(\Gamma_g^\infty)\models \sf{AD}^+.$$ This is by a theorem of Woodin, \cite[Section 3.3]{StationaryTower}, which states that if there is a proper class of Woodin cardinal and if $A\in \Gamma^\infty$, then $L(A,\mathbb{R})\models \sf{AD}^+$. Moreover, by a theorem of Steel, \cite{DMT}, every $A\in \Gamma^\infty$ has a scale in $\Gamma^\infty$, this implies that 
\begin{equation}\label{eqn:scale}
L(\Gamma^\infty)\models ``\textrm{every set of reals is Suslin}".
\end{equation}

 \textit{Sealing} is a form of Shoenfield-type generic absoluteness for the theory of universally Baire sets. Sealing is an important hypothesis in set theory and particularly in inner model theory. \cite{StrengthuB} has a detailed discussion on the importance of Sealing and related topics; so we only summarize some main points here. If a large cardinal theory $\phi$ implies Sealing then the Inner Model Program for building canonical inner models of $\phi$ cannot succeed (at least with the criteria for defining ``canonical inner models" as is done to date), cf \cite[Sealing Dichotomy]{StrengthuB}. Sealing signifies a place beyond which new methodologies are needed in order to advance the Core Model Induction techniques. In particular, to obtain consistency strength beyond Sealing from strong theories such as the Proper Forcing Axiom, one needs to construct canonical subsets of $\Gamma^\infty$ (third-order objects), instead of elements of $\Gamma^\infty$ like what has been done before (see \cite[Section 1]{StrengthuB} for a more detailed discussion). The consistency of Sealing was first proved by Woodin, who showed that if there is a proper class of Woodin cardinals and a supercompact cardinal $\k$ then Sealing holds after collapsing $2^{2^\k}$ to be countable. Woodin's proof can be found in \cite{StationaryTower}. \cite{sargsyan2024exact,sargsyan2021sealing} show that Sealing holds in hod mice and various types of hybrid mice whose existence is consistent relative to $\sf{WLW}$, which improves significantly Woodin's result. 

Woodin \cite[Theorem 3.4.17]{StationaryTower} also obtains the consistency of Tower Sealing from a supercompact cardinal and a proper class of Woodin cardinals. \cite{sargsyan2024exact} claims Tower Sealing holds in an excellent hybrid premouse (defined in \cite{sargsyan2024exact}), but this is not true. Part of this paper's motivation is to correct this, cf. Theorem \ref{thm:TSfail}. This leads us to the formulation of Partial Tower Sealing, a weak form of Tower Sealing strong enough to imply Sealing in various circumstances, cf. Theorem \ref{thm:weakTSimpliesSealing}, and weak enough to hold in hod mice. The proof that this form of Tower Sealing holds in such hod mice is given in Theorem \ref{thm:weakTS}. It is not known whether Tower Sealing can hold in hod mice at the moment.
 
%Another proof of this fact was presented in \cite{StrengthuB}, where the authors establish an actual equiconsistency for $\sf{Sealing}$. One advantage of the proof in this paper is that no smallness assumption is made (unlike \cite{StrengthuB}). Another, perhaps more important, advantage of the current proof over the one presented in \cite{StrengthuB} is that this proof is more accessible. Our proof of $\sf{Sealing}$ is based on iterability and uses recent ideas from descriptive inner model theory. However, in this paper, our aim is to present the proof of our main theorem, \rthm{main theorem}, without using any fine structure theory or heavy machinery from inner model theory, so that the paper is accessible to the widest possible audience.  We will only assume general knowledge of iterations, iteration strategies and Woodin's extender algebra, all of which are topics that can be presented without any fine structure theory. For instance, the reader can consult \cite{FarahEA} or \cite{IT}. The fact that the hypothesis of \rthm{main theorem} is weaker than a Woodin cardinal that is a limit of Woodin cardinals follows from a very recent work of Steel (\cite{normalization_comparison}) and the first author (\cite{GG}) (but also see \cite{RecentResultsIMT}), and this fact will not be proven here, as it is well beyond the scope of this paper.

\begin{theorem}\label{thm:weakTS}
Suppose $(\P,\Psi)$ is an lbr hod pair or a layered hod pair such that $\P \models $``there is a strong cardinal and a proper class of Woodin cardinals". Let $\kappa$ be the least strong cardinal of $\P$ and $g\subset Coll(\omega,\kappa^+)$ be $\P$-generic. Then $$\P[g] \models ``\forall \delta \ \textrm{ if } \delta \textrm{ is Woodin, then Partial Tower Sealing holds at } \delta."$$
\end{theorem}
\begin{remark}
In general, we cannot expect Tower Sealing to hold in generic extensions of hod mice. See Theorem \ref{thm:TSfail}.
\end{remark}

From the hypothesis of Theorem \ref{thm:weakTS} and recent work of the first author, we immediately obtain the following corollary.

\begin{corollary}\label{cor:weakTSstrength}
Partial Tower Sealing is consistent relative to $\sf{ZFC} + $``there is a Woodin limit of Woodin cardinals".
\end{corollary}

The next theorem shows that Partial Tower Sealing implies Sealing holds at certain Woodin cardinals. The reader can see section \ref{sec:prelim} and section \ref{sec:WTSimpliesS} for the definition of $Hom^*_{g*G}$ and related notions.

\begin{theorem}\label{thm:weakTSimpliesSealing}
Suppose $\delta$ is a Woodin cardinal which is a limit of Woodin  cardinals with the property that whenever $g$ is $<\delta$-generic, $G\subseteq \mathbb{Q}_{<\delta}$ is $V[g]$-generic, then $\Gamma_{g*G}^\infty = Hom^*_{g*G}$. Suppose Partial Tower Sealing holds at $\delta$. Then Sealing holds at $\delta$.

\end{theorem}

\begin{remark}\label{rem:hypoholds}
The hypothesis used in Theorem \ref{thm:weakTSimpliesSealing} holds in various important situations. For example, if $\delta$ is a Woodin limit of Woodin cardinals and strong cardinals, then whenever $G\subseteq \mathbb{Q}_{<\delta}$ is $V$-generic, $\Gamma_G^\infty = Hom^*_G$. Also, if $V$ is the universe of a hod mouse with a proper class of Woodin cardinals, then at every Woodin cardinal $\delta$ which is a limit of Woodin cardinals, whenever $G\subseteq \mathbb{Q}_{<\delta}$ is $V$-generic, $\Gamma_G^\infty = Hom^*_G$. See Section \ref{sec:WTSimpliesS} for more details.
    
\end{remark}

It is unclear whether Tower Sealing can hold in hod mice and whether Tower Sealing is consistent relative to $\sf{ZFC} + $``there is a Woodin limit of Woodin cardinals". In theorems \ref{thm:TSfail} and \ref{thm:TSfails2}, we provide further evidence that it seems very hard to force Tower Sealing to hold in hod mice. 

The paper is organized as follows. In Section \ref{sec:prelim} we review basic notions used in this paper. In Section \ref{sec:WTSimpliesS} we shows Partial Tower Sealing implies Sealing holds at certain Woodin cardinals. In Section \ref{sec:der_model}, we recall the derived model representation of $\Gamma^\infty$ in \cite{sargsyan2021sealing} and use it to prove the consistency of Partial Tower Sealing in Section \ref{sec:PartialTS}. In Section \ref{sec:TSfail} we prove Theorems \ref{thm:TSfail} and \ref{thm:TSfails2} which show that in general, Tower Sealing fails in hod mice. In Section \ref{sec:questions} we collect some open problems and questions related to the results of this paper.

\section{Preliminaries}\label{sec:prelim}
\subsection{Homogenously Suslin and universally Baire sets}
We say that a pair of trees $T,S$ are \textit{$\delta$-absolutely complementing} if for any poset $\mathbb{P}$ of size $\leq \d$, for any generic $g\subseteq \mathbb{P}$, $V[g]\models ``p[T]=\mathbb{R}-p[S]"$. Similarly, we say that $T,S$ are \textit{$<\delta$-absolutely complementing} if for any poset $\mathbb{P}$ of size $< \d$, for any generic $g\subseteq \mathbb{P}$, $V[g]\models ``p[T]=\mathbb{R}-p[S]"$. Given a limit of Woodin cardinals $\nu$ and $g\subseteq \Coll(\omega, <\nu)$, let
\begin{enumerate}
\item $\mathbb{R}^*_g=\bigcup_{\a<\nu}\mathbb{R}^{V[g\cap \Coll(\omega, \a)]}$, 
\item $Hom^*_g$ be the set of reals $A\in V(\mathbb{R}^*_g)$ such that for some $\a<\nu$, there is a pair $(T, S)\in V[g\cap \Coll(\omega, \a)]$ such that $V[g\cap \Coll(\omega, \a)]\models ``(T, S)$ are $<\nu$-complementing trees" and $p[T]^{V(\mathbb{R}^*_g)}=A$, and
\item the derived model associated with $g$ be defined by: $DM(g)=L(Hom^*_g, \mathbb{R}^*_g)$. 
\end{enumerate}

We now recall the notions of homogeneously Suslin and universally Baire sets. Given an uncountable cardinal $\kappa$, and a set $Z$, $meas_\kappa(Z)$ denotes the set of all $\kappa$-additive measures on $Z^{<\omega}$. If $\mu\in meas_\kappa(Z)$, then there is a unique $n<\omega$ such that $Z^n\in \mu$ by $\kappa$-additivity; we let this $n = dim(\mu)$. If $\mu,\nu\in meas_\kappa(Z)$, we say that \textit{$\mu$ projects to $\nu$} if $dim(\nu) = m \leq dim(\mu) = n$ and for all $A\subseteq Z^m$, 
\begin{center}
$A\in \nu \Leftrightarrow \{u : u\rest m\in A\} \in \mu$.
\end{center} 
In this case, there is a natural embedding from the ultrapower of $V$ by $\nu$ into the ultrapower of $V$ by $\mu$:
\begin{center}
$\pi_{\nu,\mu} : \Ult(V,\nu) \rightarrow \Ult(V,\mu)$
\end{center}
defined by $\pi_{\nu,\mu}([f]_\nu) = [f^*]_\mu$ where $f^*(u) = f(u\rest m)$ for all $u\in Z^n$. A tower of measures on $Z$ is a sequence $\langle \mu_n : n<k \rangle$ for some $k\leq \omega$ such that for all $m\leq n < k$, $dim(\mu_n) = n$ and $\mu_n$ projects to $\mu_m$. A tower $\langle \mu_n : n<\omega\rangle$ is \textit{countably complete} if the direct limit of $\{\Ult(V,\mu_n), \pi_{\mu_m,\mu_n}: m\leq n < \omega\}$ is well-founded. We will also say that the tower $\langle \mu_n : n<\omega\rangle$ is well-founded.

Recall we identify the set of reals $\mathbb{R}$ with the Baire space ${}^\omega \omega$.

\begin{definition}\label{dfn:homo}
Fix an uncountable cardinal $\kappa$. A function $\bar{\mu}: \omega^{<\omega} \rightarrow meas_\kappa(Z)$ is a \textit{$\kappa$-complete homogeneity system} with support $Z$ if for all $s,t\in \omega^{<\omega}$, writing $\mu_t$ for $\bar{\mu}(t)$:
\begin{enumerate}[(a)]
\item $dom(\mu_t) = dom(t)$,
\item $s\subseteq t \Rightarrow \mu_t$ projects to $\mu_s$.
\end{enumerate}
Often times, we will not specify the support $Z$; instead, we just say $\bar{\mu}$ is a $\kappa$-complete homogeneity system.

A set $A\subseteq \mathbb{R}$ is \textit{$\kappa$-homogeneous} iff there is a $\kappa$-complete homogeneity system $\bar{\mu}$ such that
\begin{center}
$A = S_{\mu} =_{def} \{ x : \bar{\mu}_x \textrm{ is countably complete} \}$.
\end{center}
$A$ is homogeneous if it is $\kappa$-homogeneous for all $\kappa$. Let $\textrm{Hom}_\infty$ be the collection of all homogeneous sets.
\end{definition}

\begin{definition}\label{dfn:ub}
$A\subseteq \mathbb{R}$ is \textit{$\kappa$-universally Baire} if there are trees $T,U \subseteq (\omega\times ON)^{<\omega}$ that are $\kappa$-absolutely complemented, i.e. $A = p[T] = \mathbb{R}\backslash p[U]$ and whenever $\mathbb{P}$ is a forcing such that $|\mathbb{P}|\leq \kappa$ and $g\subseteq \mathbb{P}$ is $V$-generic, in $V[g]$, $p[T] = \mathbb{R}\backslash p[U]$. In this case, we let $A_g = p[T]$ be the canonical interpretation of $A$ in $V[g]$.

$A$ is \textit{universally Baire} if $A$ is $\kappa$-universally Baire for all $\kappa$. Let $\Gamma^\infty$ be the collection of all universally Baire sets.
\end{definition}

We remark that if $A$ is $\kappa$-universally Baire as witnessed by pairs $(T_1, U_1)$ and $(T_2, U_2)$ and $\mathbb{P}$ such that $|\mathbb{P}|\leq \kappa$ and $g\subset \mathbb{P}$ is $V$-generic, then $A_g = p[T_1] = p[T_2]$, i.e. $A_g$ does not depend on the choice of absolutely complemented trees that witness $A$ is $\kappa$-universally Baire. A similar remark applies to $\kappa$-homogeneously Suslin sets; in other words, if $A = S_{\bar{\mu}}$ where the measures in $\bar{\mu}$ are $\kappa$-complete, then for any $<$-$\kappa$ generic $g$, the canonical interpretation $A_g$ is defined as $$(S_{\bar{\mu}})_g = \{x\in \mathbb{R}^{V[g]} : \bar{\mu}_x \textrm{ is countably complete in } V[g]\}.$$ 

Suppose there is a proper class of Woodin cardinals. The following are some standard results about universally Baire sets we will use throughout our paper. The proof of these results can be found in \cite{DMT}.

\begin{enumerate}[(I)]
\item Hom$_\infty = \Gamma^\infty$.
\item For any $A\in \Gamma^\infty$, $L(A,\mathbb{R})\models \sf{AD}^+$; furthermore, given such an $A$, there is a $B\in\Gamma^\infty$ such that $B\notin L(A,\mathbb{R})$ and $A\in L(B,\mathbb{R})$. In fact, $A^\sharp$ is an example of such a $B$. 
\item Suppose $A\in \Gamma^\infty$. Let $B$ be the code for the first order theory with real parameters of the structure $(HC,\in, A)$ (under some reasonable coding of $HC$ by reals). Then $B\in \Gamma^\infty$ and if $g$ is $V$-generic for some forcing, then in $V[g]$, $B_g\in \Gamma^\infty$ is the code for the first order theory with real parameters of $(HC^{V[g]},\in, A_g)$.
\item Every set in $\Gamma^\infty$ has a scale in $\Gamma^\infty$.
\end{enumerate}
Under the same hypothesis, the results above also imply that 
\begin{itemize}
\item $\Gamma^\infty$ is closed under Wadge reducibility, 
\item if $A\in \Gamma^\infty$, then $\neg A\in\Gamma^\infty$, 
\item if $A\in\Gamma^\infty$ and $g$ is $V$-generic for some forcing, then there is an elementary embedding $j: L(A,\mathbb{R})\rightarrow L(A_g, \mathbb{R}_g)$, where $\mathbb{R}_g = \mathbb{R}^{V[g]}$.
\end{itemize}

\subsection{Hod mice}

Suppose $(\P,\Psi)$ is a hod pair in the sense of \cite{normalization_comparison} and that $\P$ has a proper class of Woodin cardinals. We recall some properties of iteration strategies of certain countable elementary substructures $X\prec \P|\eta$ for some inaccessible $\eta$ proved in \cite{sargsyan2021sealing}. 

We adopt some notations from \cite{sargsyan2021sealing}. First, the pair $(\P,\Psi)$ is called \textit{an iterable pair} in \cite{sargsyan2021sealing}. Given a strong limit cardinal $\k$ and $F\subseteq Ord$, set 
\begin{center}
$W^{\Psi}_\k=(H_{\k}, \P|\k, \Psi_{\P|\k}\rest H_\k, \in)$.
\end{center}
Given a structure $Q$ in a language extending the language of set theory with a transitive universe, and an $X\prec Q$, we let $ M_X$ be the transitive collapse of $X$ and $\pi_X: M_X\rightarrow Q$ be the inverse of the transitive collapse. In general, the preimages of objects in $X$ will be denoted by using $X$ as a subscript, e.g. $\pi_X^{-1}(\P) = \P_X$. Also, if $\eta < \delta$, we write $\Psi_{\eta,\delta}$ for the fragment of $\Psi$ that acts on the window $(\eta,\delta)$. Suppose in addition $Q=(R,...\P,\Psi_{\eta,\delta},...)$. We will then write $X\prec (Q|\Psi_{\eta,\delta})$ to mean that $X\prec Q$ and the strategy of $\P_X$ that we are interested in is $\Psi_{\eta,\delta}^{\pi_X}$. We set $\Lambda_X=\Psi_{\eta,\delta}^{\pi_X}$. If $g$ is a generic over $V$, we write $\Psi_{\eta,\delta}^g$ for the canonical interpretation of $\Psi_{\eta,\delta}$ in $V[g]$ (if exists) and $\Lambda^g_X = (\Psi_{\eta,\delta}^g)^{\pi_X}$. 

By results of \cite{normalization_comparison}, $\Psi$ has all the properties required to run the constructions in \cite{sargsyan2021sealing}. In particular, results of \cite[Sections 2,3,4]{sargsyan2021sealing} can be applied to $(\P,\Psi)$. We summarize some key facts that we use in the constructions in Section \ref{sec:der_model}. We need develop some terminology to state these facts. In the following, we will write $V$ for the universe of $\P$ and the notions below will hold in $V$.

 Suppose $\nu$ is a Woodin cardinal. We let $\sf{EA}_\nu$ be the $\omega$-generator version of the extender algebra associated with $\nu$ (see e.g. \cite{steel2010outline} for a detailed discussion of Woodin's extender algebras).  We say the triple $(M, \d, \Phi)$ \textit{Suslin, co-Suslin captures}\footnote{This notion is probably due to Steel, see \cite{DMATM}.} the set of reals $B$ if there is a pair $(T, S)\in M$ such that $M\models ``(T, S)$ are $\d$-complementing" and 
 \begin{enumerate}
 \item $M$ is a countable transitive model of some fragment of $\sf{ZFC}$,
 \item $\Phi$ is an $\omega_1$-strategy for $M$,
 \item $M\models ``\d$ is a Woodin cardinal",
 \item for $x\in \mathbb{R}$, $x\in B$ if and only if there is an iteration $\T$ of $M$ according to $\Phi$ with last model $N$ such that $x$ is generic over $N$ for  $\textsf{EA}_{\pi^\T(\d)}^{N}$ and $x\in p[\pi^\T(T)]$.
 \end{enumerate}

Suppose $M$ is a countable transitive model of set theory and $\Phi$ is a strategy of $M$. Let $(\eta, g)$ be such that $g$ is $M$-generic for a poset in $M|\eta$. Let $\Phi'$ be the fragment of $\Phi$ that acts on iterations that are above $\eta$. Then $\Phi'$ can be viewed as an iteration strategy of $M[g]$. This is because if $\T$ is an iteration of $M[g]$ above $\eta$, there is an iteration $\U$ of $M$ that is above $\eta$ and such that 
\begin{enumerate}
\item $lh(\T)=lh(\U)$,
\item $\T$ and $\U$ have the same tree structure,
\item for each $\a<lh(\T)$, $M^\T_\a=M^\U_\a[g]$,
\item for each $\a<lh(\T)$, $E_\a^\T$ is the extension of $E_\a^\U$ onto $M_\a^\U[g]$.
\end{enumerate}
Let $\Phi''$ be the strategy of $M[g]$ with the above properties. We then say that $\Phi''$ is induced by $\Phi'$. We will often confuse $\Phi''$ with $\Phi'$. The following lemma is the key fact that we need for \ref{sec:der_model}; it is proved in \cite[Lemma 4.4]{sargsyan2021sealing}. Say $u=(\eta,\d, \lambda)$ is a \textit{good triple} if it is increasing, $\d$ is a Woodin cardinal, and $\lambda$ is an inaccessible cardinal. 

 \begin{lemma}\label{ub to strategy} Suppose $u=(\eta, \d, \lambda)$ is a good triple and $g$ is $V$-generic for a poset in $V_\eta$. Let $A\in \Gamma^\infty_g$.  Then, in $V[g]$, there is a club of countable $X\prec (W_\lambda[g]| \Psi_{\eta, \d}^g)$ such that $(M_X, \d_X, \Lambda^g_X)$ Suslin, co-Suslin captures $A$.\footnote{To conform with the above setup, we tacitly assume $\Lambda^g_X$ to be the iteration strategy acting on trees above $\eta_X$.} For each such $X$, let $X' = X\cap W_\lambda \prec W_\lambda$, and $(M_{X'}, \Lambda_{X'})$ be the transitive collapse of $X'$ and its strategy. Then $A$ is projective in $\Lambda_{X'}$. Moreover, these facts remain true in any further generic extension by a poset in $V_\eta[g]$. 
 \end{lemma}

\section{Partial Tower Sealing implies Sealing}
\label{sec:WTSimpliesS}

In this section, we prove Theorem \ref{thm:weakTSimpliesSealing} and give some applications of \ref{thm:weakTSimpliesSealing} and Partial Tower Sealing.

Suppose $\delta$ is a Woodin limit of Woodin cardinals and $G\subseteq \mathbb{Q}_{<\delta}$ is $V$-generic, we let $\mathbb{R}^*_G=\mathbb{R}^{V[G]}$ and $Hom^*_G$ be the set of $A\subseteq \mathbb{R}^*_G$ such that  for any $\gamma$, there is a homogeneity system $\bar{\mu}= \langle \mu_s : s\in \omega^{<\omega} \rangle$ with the properties:
\begin{itemize}
\item each measure $\mu_s$ is $\gamma$-complete in $V[G]$.
\item $\vec{\mu}\in V[x]$ for some $x\in \mathbb{R}^*_G$.
\item $A = \{z : (\mu_{z\rest n} : n<\omega) \textrm{ is well-founded}\}.$
\end{itemize}
In the last item above, we write $A = S_{\bar{\mu}}.$ If the measures in $\bar{\mu}$ are all $\kappa$-complete for some uncountable cardinal $\kappa$ and $g$ is $<\kappa$-generic, recall we can canonically extend $A$ to $A_g = (S_{\bar{\mu}})_g$, where $$(S_{\bar{\mu}})_g = \{z\in V[G,g] : (\mu_{z\rest n} : n<\omega) \textrm{ is well-founded in } V[G,g]\}.$$
Also, it is clear that for $\delta,G$ as above, there is a $\gamma$ (sufficiently large) such that any $A\in Hom^*_G$ is witnessed by a homogeneity system $\bar{\mu}$ where each measure $\mu\in\bar{\mu}$ is $\gamma$-complete.

\begin{remark}\label{rmk:UB}
Let $\delta,G,A$ be as above. Suppose further that there is a proper class of Woodin cardinals. We can in fact choose $\gamma$ large enough so that letting $g\in V(\mathbb{R}^*_G)$ be $< \delta$-generic such that $\bar{\mu}\in V[g]$ consists of $\gamma$-complete measures, then $(S_{\bar{\mu}})_g$ is in fact universally Baire in $V[g]$ and that $A = (S_{\bar{\mu}})_G$ is universally Baire in $V[G]$. This follows from results of Woodin and Martin-Steel, \cite{DMT}. See item (I) of Section \ref{sec:prelim}.
\end{remark}

\begin{proof}[Proof of Theorem \ref{thm:weakTSimpliesSealing}]  We fix a Woodin cardinal $\delta$ as in the hypothesis of the theorem. We need to verify clause (2) of Definition \ref{dfn:ub_sealing}.  Let $\mathbb{P}\in V_\delta$ and $g\subseteq\mathbb{P}$ be $V$-generic. We show that there is an elementary embedding $j:L(\Gamma^\infty,\mathbb{R})\rightarrow L(\Gamma_g^\infty,\mathbb{R}_g)$. Even though this only proves a special case of clause (2), but it will be evident that the proof can be generalized to prove the full statement. To simplify the notation, we write $L(\Gamma^\infty)$ for $L(\Gamma^\infty,\mathbb{R})$ etc.

Let $G\subseteq \mathbb{Q}_{<\delta}$ be $V$-generic and $j_G$ be the associated embedding.  By Partial Tower Sealing, we have an elementary map $l: L(\Gamma^\infty_G)\rightarrow L(j_G(\Gamma^\infty))$ such that $l\rest \Gamma_G^\infty = \textrm{id}$. We can then define a map $$i: L(\Gamma^\infty) \rightarrow L(\Gamma_G^\infty)$$ as the unique map determined by: $i(A) = A_G$ for each $A\in \Gamma^\infty$ and $i(\alpha)=\beta$ iff $l(\beta) = j_G(\alpha)$ for any indiscernible $\alpha$ of $L(\Gamma^\infty)$. It is easy to see that $i$ can be canonically extended to all of $L(\Gamma^\infty)$ because every $z\in L(\Gamma^\infty)$ has the form $\tau[x,A,s]$ for some term $\tau$, a real $x$, $A\in\Gamma^\infty$, and $s$ a finite set of indiscernibles. Therefore, we can simply define $$i(z) = \tau^{L(\Gamma^\infty_G)}[x,A_G,i(s)].$$

\begin{claim}\label{claim:elem}
$i$ is well-defined and elementary.

\end{claim}
\begin{proof}
Suppose $\varphi$ is a formula, $x\in \mathbb{R}$, $A\in\Gamma_\infty$, and $s$ is a finite sequence of indiscernibles.  Then

\begin{align*}
L(\Gamma^\infty)\models \varphi[A,x,s] & \Leftrightarrow L(j_G(\Gamma^\infty)) \models \varphi[j_G(A), x, j_G(s)]\\
             & \Leftrightarrow L(\Gamma_G^\infty)\models \varphi[A_G,x, i(s)].
\end{align*}
The first equivalence follows from elementarity of $j_G$ and the second equivalence follows from the elementarity of $l$ and the fact that $l(i(s)) = j_G(s)$ and $l(A_G) = j_G(A) = A_G$. It is easy to see that the equivalences above prove the claim.
\end{proof}

Let $G'\subseteq \mathbb{Q}_{<\delta}$ be $V[g]$-generic and $j_{G'}: V[g]\rightarrow M'\subseteq V[g][G']$ be the associated embedding. By Partial Tower Sealing, we have an elementary map $l': L(\Gamma^\infty_{g*G'})\rightarrow L(j_{G'}(\Gamma^\infty_g))$ such that $l'\rest \Gamma^\infty_{g*G'} = \textrm{id}$. As before, we can define the elementary map $i': L(\Gamma_g^\infty)\rightarrow L(\Gamma_{g*G'}^\infty)$ similar to how $i$ was defined.

Now, we can find $G,G'$ such that the following hold:
\begin{enumerate}[(i)]
\item $G\subseteq \mathbb{Q}_{<\delta}$ is $V$-generic such that $g\in V[G]$.
\item $G'\subseteq \mathbb{Q}_{<\delta}$ is $V[g]$-generic.
\item $\mathbb{R}^{V[G]} = \mathbb{R}^{V[g][G']}$.
\item $\Gamma_G^\infty = \Gamma_{g*G'}^\infty$.
\end{enumerate}

By standard facts concerning $\mathbb{Q}_{<\delta}$, we can find $G,G'$ satisfying (i)-(iii), see \cite{DMT}. We give a little more details here. By the usual factoring property of Coll$(\omega,<\delta)$ and the fact that $g$ is $<\delta$-generic, there is a $V$-generic $H\subseteq \textrm{Coll}(\omega,<\delta)$ such that $g\in V[H]$. By \cite[Lemma 6.6]{DMT}, there are generics $G\subset \mathbb{Q}_{<\delta}$ and $G'\subset \mathbb{Q}_{<\delta}^{V[g]}$ such that $G$ is $V$-generic, $G'$ is $V[g]$-generic such that $\mathbb{R}^{V[G]} = \mathbb{R}^{V[g][G']} =\mathbb{R}^{V[H]}$. These generics $G,G'$ clearly satisfy (i) - (iii).

We show that (iv) is satisfied as well. In fact, we show 
\begin{equation}\label{eqn:4}
\Gamma_G^\infty = \Gamma_{g*G'}^\infty = Hom^*_{g*G'}=Hom^*_G.
\end{equation}
First, note that $Hom^*_{g*G'}=Hom^*_G$. This is because $\mathbb{R}^*_G = \mathbb{R}^*_{g*G'}$, so any homogeneity system $\bar{\mu}$ witnessing $A\in Hom^*_G$ is in $V[x]$ for some $x\in \mathbb{R}^*_G = \mathbb{R}^*_{g*G'}$; therefore, $\bar{\mu}$ witnesses $A\in Hom^*_{g*G'}$. The converse is proved the same way. 

But then note that by our hypothesis: 
\begin{equation*}%\label{eqn:HomequalGamma}
Hom^*_G = \Gamma_G^\infty
\end{equation*}
and
\begin{equation*}%\label{eqn:HomequalGamma2}
Hom^*_{g*G'} = \Gamma_{g*G'}^\infty.
\end{equation*}
So the equalities in (\ref{eqn:4}) hold.
%We just prove \ref{eqn:HomequalGamma} as the proof of \ref{eqn:HomequalGamma2} is the same. Let $A\in Hom^*_G$. Then there is some $\eta < \delta$ and some $B\in Hom^{V[G\rest \eta]}_{<\delta}$ such that $B^* = A$, where $B^*$ is the canonical extension of $B$ to $V[G]$. Since $B$ is $<$-$\delta$-universally Baire and $\delta$ is a limit of strong cardinals, $B$ is universally Baire in $V[G\rest \eta]$. But then $B^* = A$ is universally Baire in $V[G]$, i.e. $A\in \Gamma^G_\infty$. Conversely, suppose $A\in \Gamma^G_\infty$. Let $\sigma$ be a countable set of measures witnessing $A$ is $\kappa$-homogenously Suslin for some $\kappa >> \delta$. Since $\sigma$ is countable and $\delta = \omega_1^{V[G]}$, we may assume $\sigma\in V[G\rest \eta]$ and is countable there for some $\eta < \delta$. Let $B = S_\sigma$ be the $\kappa$-homogeneously Suslin set in $V[G\rest \eta]$ witnessed by $\sigma$. So $A = S_\sigma^{V[G]}$ is the canonical extension of $B$. Also, since $\delta$ is a limit of Woodin cardinal and $B$ is $\kappa$-homogeneously Suslin for $\kappa > \delta$, by the Martin-Steel's theorem and Woodin's theorem \cite{DMT}, $B\in Hom^{V[G\rest\eta]}_{<\delta}$. This means $A\in Hom^*_G$ as desired.

Let $G,G'$ satisfy (i)-(iv) above. So we have elementary embeddings $i_0: L(\Gamma^\infty)\rightarrow L(\Gamma_G^\infty) = L(\Gamma_{g*G'}^\infty)$ and $i_1: L(\Gamma_g^\infty)\rightarrow L(\Gamma_G^\infty) = L(\Gamma_{g*G'}^\infty)$, where $i_0 = j_G\rest L(\Gamma^\infty)$ and $i_1 = j_{G'}\rest L(\Gamma_g^\infty)$. Let $j: L(\Gamma^\infty)\rightarrow L(\Gamma_g^\infty)$ be defined by: $j(A) = A_g$ for each $A\in \Gamma^\infty$ and $j(\alpha)=\beta$ iff $i_1(\beta)=i_0(\alpha)$. Just like in the proof of Claim \ref{claim:elem}, we have that $k$ is elementary. This completes the proof of the theorem.

\end{proof}

\begin{corollary}\label{cor:hypoholds}
Suppose there is a proper class of Woodin cardinals and $\delta$ is a Woodin limit of Woodin cardinals. Then whenever $g$ is $<$-$\delta$ generic, whenever $G\subseteq \mathbb{Q}_{<\delta}$ is $V[g]$-generic, then $\Gamma_{g*G}^\infty = Hom^*_{g*G}$. Therefore, if Partial Tower Sealing holds at $\delta$, then Sealing holds at $\delta$.
 %\item[(ii)] Suppose $V$ is the universe of a hod mouse with a proper class of Woodin cardinals. Suppose $\delta$ is Woodin limit of Woodin cardinals and whenever $g$ is $<$-$\delta$ generic, whenever $G\subseteq \mathbb{Q}_{<\delta}$ is $V[g]$-generic, then $\Gamma_{g*G}^\infty = Hom^*_{g*G}$. Therefore, if Partial Tower Sealing holds at $\delta$, then Sealing holds at $\delta$.
\end{corollary}
\begin{proof}
Let $g,G$ be as in the statement of the corollary. As in the proof of Theorem \ref{thm:weakTSimpliesSealing}, we can find $G'\subseteq \mathbb{Q}_{<\delta}$ such that 

\begin{enumerate}[(a)]
\item $G'\subseteq \mathbb{Q}_{<\delta}$ is $V$-generic such that $g\in V[G']$.
\item $G'\subseteq \mathbb{Q}_{<\delta}$ is $V[g]$-generic.
\item $\mathbb{R}^{V[G']} = \mathbb{R}^{V[g][G]}$.
%\item $\Gamma^{G'}_\infty = \Gamma^{g*G}_\infty$.
\end{enumerate}
We need to verify that 
\begin{enumerate}
\item[(d)] $\Gamma_{G'}^\infty = \Gamma_{g*G}^\infty$.
\end{enumerate}
As in the proof of (\ref{eqn:4}), $Hom^*_{g*G}=Hom^*_{G'}$. We need to verify:
\begin{equation}\label{eqn:HomequalGamma}
Hom^*_{G'} = \Gamma_{G'}^\infty
\end{equation}
and
\begin{equation}\label{eqn:HomequalGamma2}
Hom^*_{g*G} = \Gamma_{g*G}^\infty.
\end{equation}
We just prove \ref{eqn:HomequalGamma} as the proof of \ref{eqn:HomequalGamma2} is the same. Let $A\in Hom^*_{G'}$. Then there is some $g\in V(\mathbb{R}^*_{G'})$ such that $g$ is $<\delta$-generic and some $\bar{\mu}\in V[g]$ such that $A = (S_{\bar{\mu}})_{G'}$. Let $B=(S_{\bar{\mu}})_g$. We may assume all measures in $\bar{\mu}$ are $\gamma$-complete for a sufficiently large $\gamma$ so that $B$ is in fact universally Baire in $V[g]$ and that $A$ is universally Baire in $V[G']$ (see Remark \ref{rmk:UB}). So $A\in \Gamma_{G'}^\infty$. Conversely, suppose $A\in \Gamma_{G'}^\infty$. By work of Martin-Steel and Woodin \cite{DMT} and the fact that there is a proper class of Woodin cardinals, $A$ is $\kappa$-homogeneously Suslin for some sufficiently large $\kappa > \delta$. Let $\bar{\mu}$  witness $A$ is $\kappa$-homogenously Suslin. Since $\bar{\mu}$ is countable and $\delta = \omega_1^{V[G']}$, $\bar{\mu}\in V[g]$ for some $<\delta$ generic $g\in V(\mathbb{R}^*_{G'})$. Let $B = (S_{\bar{\mu}})_g$ be the $\kappa$-homogeneously Suslin set in $V[g]$ witnessed by $\bar{\mu}$. So $A = (S_{\bar{\mu}})_{G'}$ is the canonical extension of $B$.  This means $A\in Hom^*_{G'}$ as desired.\footnote{We may assume $\kappa$ is large enough that $\bar{\mu}$ witnesses $A\in Hom^*_{G'}.$}

%(ii) follows from the arguments in (i) as this is a special case of (i). 

%$\Gamma_{g*G}^\infty$ is the closure of $(\Q,\Lambda^{g*G})$ where $\Q = V|\alpha$ for $\alpha<\delta$ and $\Lambda$ is the canonical strategy for $\Q$ by results of \cite{hod_mice_LSA, normalization_comparison}). It is clear that $\Gamma_{g*G}^\infty\subseteq Hom^*_{g*G}$. So in fact we obtain $\Gamma_{g*G}^\infty = Hom^*_{g*G}$.

\end{proof}

Assume $\sf{AD}^+$, we say that a pointclass $\Gamma$ such that $\Gamma = \powerset(\mathbb{R})\cap L(\Gamma)$ is $OD$-full if whenever $A\in \Gamma$ and $x,y\in\mathbb{R}$ are such that $y\in OD(A,x)$, then $y\in OD(A,x)$ in $L(\Gamma)$. We write $\Theta$ for the supremum of ordinals $\alpha$ for which there is a surjection of $\mathbb{R}$ onto $\alpha$; we write $(\theta_\alpha : \alpha \leq \gamma)$ for the Solovay sequence. These notations can be relativized to pointclasses like $\Gamma$ and we write $\Theta^\Gamma, \theta_\alpha^\Gamma$ for such objects. For $\Sigma$ an $(\omega_1,\omega_1)$-iteration strategy for a countable mouse or a hod mouse $\P$, for $a\in \textrm{HC}$, we let Lp$^\Sigma(a)$ be the stack of all sound $\Sigma$-mice $\M$ over $a$ such that $\rho_\omega(\M) = \omega$.

\begin{theorem}\label{prop:Gammafull}
Suppose Partial Tower Sealing holds at a Woodin cardinal $\delta$. Let $G\subseteq \mathbb{Q}_{<\delta}$ be $V$-generic and $j_G: V\rightarrow M\subseteq V[G]$ be the associated embedding, then $\Gamma^G_\infty$ is $OD$-full in $j_G(\Gamma_\infty)$. In particular, the following hold.
\begin{enumerate}
\item Suppose $\Sigma\in \Gamma^\infty_G$ is an iteration strategy, then for any $a\in HC^{V[G]}$, Lp$^\Sigma(a)\cap L(\Gamma^\infty_G) = \textrm{Lp}^\Sigma(a)\cap L(j_G(\Gamma^\infty))$.
\item $\Theta^{\Gamma^\infty_G} = \theta_\alpha^{j(\Gamma^\infty)}$ for some limit ordinal $\alpha.$
\end{enumerate}
\end{theorem}
\begin{proof}
Fix $G,j_G$ and let $l: L(\Gamma^G_\infty)\rightarrow L(j_G(\Gamma^\infty))$ be given by Partial Tower Sealing. Let $a\in HC^{V[G]}$ and let $A\in \Gamma_G^\infty$. Suppose $y\in OD(A,a)\cap HC^{V[G]}$ in $L(\Gamma_G^\infty)$, then by elementarity and the fact that $l(A) = A$ and $l\rest \mathbb{R}^{V[G]} = \textrm{id}$, we have that $y\in OD(A,a)$ in $L(j_G(\Gamma^\infty))$. Conversely, if $y\in OD(A,a)\cap HC^{V[G]}$ in $L(j(\Gamma^\infty))$, then letting $T$ be the tree projecting to the universal $\undertilde{\Sigma^2_1}(A)$-set in $L(j_G(\Gamma^\infty))$; the existence of $T$ follows from the fact that $$L(j_G(\Gamma^\infty)) \models ``\textrm{ there is no largest Suslin cardinal}"$$ which follows from (\ref{eqn:scale}). So $y\in L[T,a]$. Now by our choice of $T$ and the fact that $\undertilde{\Sigma^2_1}(A)$ is the same in $L(j_G(\Gamma^\infty))$ and in $L(\Gamma^\infty_G)$, $T\in L(\Gamma^\infty_G)$. We have then that $l(T) = T$. Since $L(\Gamma_G^\infty)\models y\in L[T,a]$, we see that $y\in OD(A,a)$ in $L(\Gamma_G^\infty)$.

To see (1), note that by elementarity and the fact that $l(\Sigma)=\Sigma$, $$l(\textrm{Lp}^\Sigma(a)^{L(\Gamma_G^\infty)})= (\textrm{Lp}^\Sigma(a))^{L(j_G(\Gamma^\infty))}.$$ Clearly $\textrm{Lp}^\Sigma(a)^{L(\Gamma_G^\infty)}\unlhd \textrm{Lp}^\Sigma(a))^{L(j_G(\Gamma^\infty))}$. Now suppose $\M\lhd (\textrm{Lp}^\Sigma(a))^{L(j_G(\Gamma^\infty))}$ is the least that is not in $(\textrm{Lp}^\Sigma(a))^{L(\Gamma_G^\infty)}$, then since $\M\in OD(\Sigma,a)$ in $L(j_G(\Gamma^\infty))$, $\M\in OD(\Sigma,a)$ in $L(\Gamma_G^\infty)$. Since $L(j_G(\Gamma^\infty))\models \M$ is a $\Sigma$-mouse over $a$, using $l$, we see that $$L(\Gamma^\infty_G)\models \M \textrm{ is a $\Sigma$-mouse over $a$}.$$ This means $\M\lhd (\textrm{Lp}^\Sigma(a))^{L(\Gamma^\infty_G)}$ as desired.

For (2), the fact that $\Theta^{\Gamma_G^\infty} = \theta_\alpha^{j_G(\Gamma^\infty)}$ follows from $OD$-fullness of $\Gamma_G^\infty$; in fact, for each $\beta$ such that $\theta_\beta < \theta_\alpha$ in $j(\Gamma^\infty)$, $\theta_\beta^{\Gamma_G^\infty}=\theta_\beta^{j(\Gamma^\infty)}$. $\alpha$ is limit because $L(\Gamma_G^\infty)\models$ ``every set of reals is Suslin" as mentioned above.
\end{proof}

\section{Derived model representation of $\Gamma^\infty$ and Sealing}
\label{sec:der_model}
In this section, we summarize the construction in \cite{sargsyan2021sealing} that realizes $\Gamma_\infty$ as a derived model via a direct limit construction. We assume the hypothesis of Theorem \ref{thm:weakTS} and write $V$ for the universe of $\P$. We fix a generic $g\subseteq \textrm{Coll}(\omega,\kappa^+)$ and write $\iota$ for $\kappa^+$. We note that by our assumption and standard theory of hod mice \cite{normalization_comparison}, the hypotheses required to apply Theorem 0.4 of \cite{sargsyan2021sealing} are satisfied in $V[g]$.

We say $u=(\eta, \d, \d', \lambda)$ is a \textit{good quadruple} if $(\eta, \d, \lambda)$ and $(\eta, \d', \lambda)$ are good triples with $\d<\d'$ (see Section \ref{sec:prelim}).  Suppose $u=(\eta, \d, \d', \lambda)$ is a good quadruple and $h$ is a $V[g]$-generic such that $g*h$ is generic for a poset in $V_\eta$.  Working in $V[g*h]$, let $D(h, \eta, \d, \lambda)$ be the club of countable 
\begin{center}
$X\prec ((W_\lambda[g*h], u)| \Psi^g_{\eta, \d})$
\end{center}
 such that $H_{\iota}^V\cup\{g\}\subseteq X$. 

 Suppose $A\in \Gamma^\infty_{g*h}$. Then for a club of $X\in D(h, \eta, \d, \lambda)$, $A$ is Suslin, co-Suslin captured by $(M_X, \d_X, \Lambda^{g*h}_X)$ and $A$ is projective in $\Lambda_{X'}$ where $X'  =X\cap W_\lambda$ (see \rlem{ub to strategy}). Given such an $X$, we say $X$ \textit{captures} $A$.

Let $k\subseteq \Coll(\omega, \Gamma^\infty_{g*h})$ be generic, and let $(A_i:i<\omega),(w_i: i<\omega)$ be generic enumerations of $\Gamma^\infty_{g*h}$ and $\mathbb{R}_{g*h}$ respectively in $V[g*h*k]$. Let $(X_i: i<\omega)\in V[g*h*k]$ be such that for each $i$
\begin{enumerate}
\item $X_i \in D(h,\eta,\delta,\lambda)$, and
%\item $X_i\in X_{i+1}$, and
%\item letting $H\subseteq Coll(\omega, <\eta_{X_{i+1}})$ be $M_{X_{i+1}}$-generic and $\T\in M_{X_{i+1}}[H]$ be a countable tree on $M_{X_i}$ according to $\Lambda_{X_i}$, $\T$ is $\pi_{X_{i+1}}^{-1}(\pi_{X_i})$-realizable\footnote{This can be arranged by letting $X_i$ be the restriction of some $Y\prec H_\l[g*h]$.}, 
\item $X_i$ captures $A_i$. 
\end{enumerate}
 In particular, $A_i$ is projective in $\Lambda_{X^{'}_i}$, where $X^{'}_i = X_i\cap W_\lambda$.
We set  $M^0_n=M_{X'_n}$, 
 $\pi^0_n=\pi_{X_0}$, 
$\k_0=\k_{X_0}$,
 $\nu_0=\d_{X_0}$,
 $\nu_0'=\d_{X_0}'$,
$\eta_0=\eta_{X_0}$,
 $\d_0=\d$,
$\P_0=\mathcal{P}$.

Next we inductively define sequences $(M^i_n: i, n<\omega)$, $(\pi^i_n: i, n<\omega)$, $(\Lambda_i: i\leq \omega)$, $(\tau^{i, i+1}_n: i, n<\omega)$, $(\nu_n: i<\omega)$, $(\nu_n': i<\omega)$, $(\eta_n: n<\omega)$, $(\k_i:i<\omega)$, $(\theta_i: i<\omega)$, $(\T_i, E_i: i<\omega)$, $(M_i': i<\omega)$, $(\U_i, F_i: i<\omega)$, $(\P_i: i\leq n)$, $(\P_i': i<\omega)$, and $(\sigma_i: i<\omega)$ satisfying the following conditions (see Figure 5.1 of \cite{sargsyan2021sealing}).
\begin{enumerate}[(a)]
\item For all $i, n<\omega$, $\pi^i_n: M^i_n\rightarrow \P_i$ and $rng(\pi^i_n)\subseteq rng(\pi^i_{n+1})$.
\item $\tau^{i, i+1}_n: M^i_n\rightarrow M^{i+1}_n$. Let $\tau_n: M^0_n\rightarrow M^n_n$ be the composition of $\tau^{j, j+1}_n$'s for $j<n$. 
\item For all $i, n<\omega$, $\k_n=\tau_n(\k_0)$, $\eta_n=\tau_n(\eta_{0})$, $\nu_n=\tau_n(\nu_0)$ and $\nu_n'=\tau_n(\nu_0')$. 
\item For all $n<\omega$, $\T_n$ is an iteration of $M^n_n|\nu_n'$ above $\nu_n$ that makes $w_n$ generic and $M_n'$ is its last model. 
\item $\theta_n=\pi^{\T_n}(\nu_n')$ and $E_n\in \vec{E}^{M_n'}$ is such that $lh(E_n)>\theta_n$ and $\cp(E_n)=\k_n$.
\item for all $m,n$, $M^{n+1}_m=Ult(M^n_m, E_n)$ and $\tau_m^{n, n+1}=\pi_{E_n}^{M^n_m}$.
\item $\U_n=\pi^n_n\T_n$, $\P_n'$ is the last model of $\U_n$, $\sigma_n: M_n'\rightarrow \P_n'$ is the copy map and $F_n=\sigma_n(E_n)$.\footnote{So $\oplus_{i\leq n}\T_i$ and $\oplus_{i\leq n} \U_i$ are sealed iterations based on $\kappa$.}
\item $\P_{n+1}=Ult(\P_n, F_n)$ and $\psi^{n+1}_m:M^{n+1}_m\rightarrow \P_{n+1}$ is given by $\pi^{n+1}_m(\pi_{E_n}^{M^n_m}(f)(a))=\pi_{F_n}^{\P_n}(\pi^n_m(f))(\sigma_n(a))$.
\item $\Lambda_n=(\pi^n_n$-pullback of $(\Psi^{g*h}_\lambda)_{\P_n|\psi_n(\nu_n)})_{\eta_n, \nu_n}=(\sigma_n$-pullback of $(\Psi^{g*h}_\lambda)_{\P_n'|\sigma_n(\nu_n)})_{\eta_n, \nu_n}$ (see \cite[Corollary 3.6]{sargsyan2021sealing}).
\end{enumerate}

Let $M^{\omega}_{n}$ be the direct limit of $(M^m_n: m<\omega)$ under the maps $\tau^{m, m+1}_n$. Letting $\P_\omega$ be the direct limit of $(\P_n: n<\omega)$ and the compositions of $\pi_{F_n}^{\P_n}$, we have natural maps $\pi^\omega_n:M^\omega_n\rightarrow \P_\omega$. Notice that\\\\
(1) for each $n<\omega$, $\k_n<\omega_1^{V[g*h]}$ and $sup_n\k_n=\omega_1^{V[g*h]}$.\\\\
It follows that if $\tau^m_n: M_n^m\rightarrow M^\omega_n$ is the direct limit embedding then\\\\
(2) $\tau^m_n(\k_n)=\omega_1^{V[g*h]}$. \\\\
Next, notice that\\\\
(3) for each $m, n, p$, letting $\iota_n=\tau_n(\iota_{X_0})=\tau_n(\iota)$, $M^n_m|\iota_n=M^n_p|\iota_n$ and $\iota_n=(\k_n^+)^{M^n_m}$.\\
(4) for each $m,n, p$, $\pi^n_m\rest (M^n_m|\iota_n)=\pi^n_p\rest (M^n_p|\iota_n)$\\
(5) for each $m$, $n>1$ and $p>n$, $M^n_m|\theta_{n-1}=M^p_m|\theta_{n-1}$.\\
(6) for each $m$, $n>1$ and $p$ with $p>n$, $\pi^n_m\rest (M^n_m|\theta_{n-1})=\pi^p_m\rest (M^p_m|\theta_{n-1})$.\\\\
Because of condition (d) above we can find $G\subseteq Coll(\omega, <\omega_1^{V[g*h]})$ generic over $M^\omega_n$ (for each $n<\omega$) such that $\mathbb{R}^{M^\omega_n[G]}=\mathbb{R}_{g*h}$ and $G\in V[g*h*k]$. By constructions, $\omega_1^{V[g*h]}$ is a limit of Woodin cardinals in $M^\omega_n$. \cite{sargsyan2021sealing} shows that

\begin{lemma}\label{der model rep} For each $n<\omega$, $DM(G)^{M^\omega_n[G]}=L(\Gamma^\infty_{g*h}, \mathbb{R}_{g*h})$. 
\end{lemma}

\rlem{der model rep} implies that clause 1 of Sealing and of Partial Tower Sealing holds. \cite{sargsyan2021sealing} uses \rlem{der model rep} to also verify clause 2 of Sealing holds.

In the next section, we will use the above constructions to verify clause 2 of Partial Tower Sealing holds. We say that the sequence $(X_i: i<\omega)$ is cofinal in $\Gamma^\infty_{g*h}$ as witnessed by $(A_i: i\in \omega)$ and $(w_i: i<\omega)$. We also say that $(M^n_0, \Lambda_n, \theta_n, \tau_{n, m}: n<m<\omega)$ is a $\Gamma^\infty_{g*h}$-genericity iteration induced by $(X_i: i<\omega)$ where $\tau_{n, m}: M^n_0\rightarrow M^m_0$ is the composition of $\tau^{i, i+1}_0$ for $i\in [n, m)$.

\section{Partial Tower Sealing}\label{sec:PartialTS}

In this section, we use the results of the previous section to prove Theorem \ref{thm:weakTS}. We work in the universe of $\P$, which we call $V$. Let $\kappa$ be the least strong cardinal. Let $g\subseteq \textrm{Coll}(\omega,\kappa^+)$ be $V$-generic and let $\delta > \kappa$ be Woodin. We prove Partial Tower Sealing holds in $V[g]$ at $\delta$.

Work in $V[g]$, let $G\subseteq \mathbb{Q}_{<\delta}$ be $V[g]$-generic and let $j_G: V[g] \rightarrow M \subseteq V[g,G]$ be the associated embedding. We want to find an embedding $j: L(\Gamma_{g*G}^\infty)\rightarrow L(j_G(\Gamma_g^\infty))$ such that $j\rest \Gamma_{g*G}^\infty$ is the identity and furthermore, $j$ is an order-preserving bijection on the class of indiscernibles of the models.

We note that the main result of \cite{sargsyan2021sealing} already shows Sealing holds in $V[g]$ at $\delta$, therefore, there is an elementary embedding $i: L(\Gamma_g^\infty)\rightarrow L(\Gamma_{g*G}^\infty)$ such that $i(A) = A_G$ for all $A\in \Gamma_g^\infty$. Furthermore, $(\Gamma_g^\infty)^\sharp, (\Gamma_{g*G}^\infty)^\sharp$ exist and $i$ is the order-preserving bijection on the class of indiscernibles of the models.

$j_G$ induces an elementary embedding $k: L(\Gamma_g^\infty)\rightarrow L(j_G(\Gamma_g^\infty))$ such that $j_G(A) = A_G$ for all $A\in\Gamma_g^\infty$. If $\Gamma_{g*G}^\infty = j_G(\Gamma_g^\infty)$ then we simply let $j$ be the identity. In general, let $W = L(\Gamma_{g*G}^\infty)$, $W' = L(j_G(\Gamma_g^\infty))$, $\tau$ a term, $A\in \Gamma_{g*G}^\infty$, $x\in \mathbb{R}_{g*G}$, and $s$ a finite sequence of indiscernibles for both $W, W'$ such that $j_G(s) = s$ and $i(s) = s$, then we define $$j(\tau^W(A,x,s)) = \tau^{W'}(A,x,s).$$ Since $(\Gamma_{g*G}^\infty)^\sharp, (j_G(\Gamma_g^\infty))^\sharp$ exist, everything in $W$ has the form $\tau^W(A,x,s)$ for some $A,x,s$ (and similarly for $W'$), $j$ is defined on all of $W$. We need to check that $j$ is elementary. 

Let $(\xi_i : i<\omega)$ be the first $\omega$ indiscernibles for both $W,W'$ with the properties described above; we may assume that $s = (\xi_i : i< lh(s))$. Let $u=(\eta, \d', \d'', \lambda)$ be a good quadruple such that $\sup_{i<\omega}\xi_i<\eta$. Let $k\subseteq \Coll(\omega, \Gamma^\infty_{g*G})$ be $V[g*G]$-generic and $k'\subseteq Coll(\omega, j_G(\Gamma^\infty_g))$ be $M$-generic. We may assume $k'\in V[g*G*k]$

We have that $\Gamma^\infty_{g*G}$ is the Wadge closure of strategies of the countable substructures of $V[g]_\lambda$. More precisely, given $A\in \Gamma^\infty_{g*G}$, there is an $X\prec (W_\lambda|\Psi_{\eta, \d'}^{g*G})$ such that $A$ is Wadge reducible to $\Lambda_X$ and $\Lambda_X\in \Gamma_{g*G}^\infty$. It follows that to show that $j$ is elementary it is enough to show that given a formula $\phi$, $m\in \omega$, $u_m$ being the first $m$ common indiscernbiles of $W,W'$ that are fixed points of all relevant embeddings, $X\prec ((V[g]_\lambda, u)|\Psi_{\eta, \d'}^{g*G})$ and a real $x\in \mathbb{R}_{g*G}$,
\begin{center}
$W\models \phi[u_m, \Lambda_X, x]\Rightarrow W'\models \phi[u_m, \Lambda_X, x]$.\footnote{The $\Leftarrow$ is similar as will be evident by the following proof.}
\end{center}
Fix then a tuple $(\phi, n, X, x)$ as above.

Working inside $V[g*G*k]$, let $(Y_i: i<\omega)$ be a cofinal sequence in $\Gamma^\infty_{g*G}$ as witnessed by some $\vec{A}$ and $\vec{w}$ such that $A_0=\emptyset$, $w_0=x$ and $Y'_0=X$. Using $k'$, we also construct $(Z_i: i<\omega)$, a cofinal sequence in $j_G(\Gamma_g^\infty)$ as witnessed by some $\vec{B}$ and $\vec{v}$ such that $B_0=\emptyset$, $v_0=x$ and $Z'_0=X$. Here the $Z_i$'s are elementary substructures of $V^M_{j_G(\lambda)}$.

Let $(M^n_0, \Lambda_n, \theta_n, \tau_{n, l}: n<l<\omega)$ be a $\Gamma^\infty_{g*G}$-genericity iteration induced by $(Y_i: i<\omega)$ and $(N^n_0, \Phi_n, \nu_n, \sigma_{n, l}: n<l<\omega)$ be a $j_G(\Gamma_g^\infty)$-genericity iteration induced by $(Z_i: i<\omega)$. It is not hard to see that we can make sure that $M^1_0=N^1_0$ by simply selecting the same extender $E_0$ after $\T_0$; by our assumptions, $M^0_0 = N^0_0$ and $w_0 = v_0$. Note that this makes sense because if $X\prec V[g]_\lambda$, then $X\prec V^M_{j_G(\lambda)}$ since $V[g]_\lambda\prec V^M_{j_G(\lambda)}$. Furthermore, by elemenarity of $j_G$ and the fact that $\omega_1^{V[G]} > \iota > \kappa$, $M = j_G(V)[g]$ and $j(V) \models ``\kappa $ is strong". See Figure \ref{Two iterations}. 

$L(\Gamma^{g*G}_\infty)$ is realized as the derived model of $M^\omega_0$ and $L(j_G(\Gamma^g_\infty))$ is realized as the derived model of $N^\omega_0$, and the two iterations agree on the first extender used (namely $E_0$) and therefore $M^1_0 = N^1_0$.

\begin{figure}
\centering
\resizebox{0.4\textheight}{!}{
\begin{tikzpicture}[align=center, node distance= 2.5cm, auto]
\node (A) {$M^0_0=N^0_0$};
\node (B) [right of=A] {$M^1_0=N^1_0$};
\node (C) [right of=B]{};
\node (D) [above of=C]{$M^2_0$};
\node (E) [below of=C]{$N^2_0$};
\node (F) [right of=D]{$M^3_0$};
\node (G) [right of=E]{$N^3_0$};
\node (FF) [right of=F]{};
\node (GG) [right of=G]{};
\node (HH) [right of=FF]{};
\node (II) [right of=GG]{};
\node (H) [right of=HH]{$M^\omega_0$};
\node (I) [right of=II]{$N^\omega_0$};
\draw[->] (A) to node {}(B);
\draw[->] (B) to node {}(D);
\draw[->] (B) to node {}(E);
\draw[->] (D) to node {}(F);
\draw[->] (E) to node {}(G);
\draw[->] (F) to node {}(FF);
\draw[->] (G) to node {}(GG);
\draw[->] (HH) to node {}(H);
\draw[->] (II) to node {}(I);
\path (FF)--(HH) node [font=\Huge, midway, sloped]{${}_{\ldots}$};
\path (GG)--(II) node [font=\Huge, midway, sloped]{${}_{\ldots}$};
\end{tikzpicture}
}
\caption{Two genericity iterations: the derived mode of $M^\omega_0$ is $L(\Gamma_{g*G}^\infty)$ and the derived model of $N^\omega_0$ is $L(j_G(\Gamma_g^\infty))$. }
\label{Two iterations}
\end{figure}
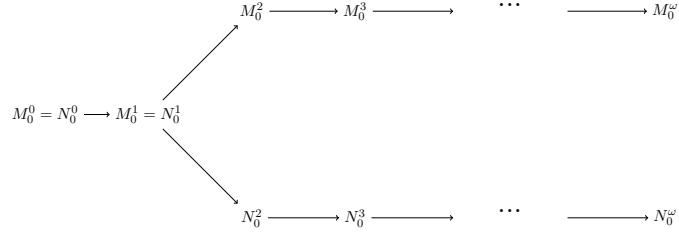

Let $\zeta=\eta_X$ and $\Gamma=(\Psi_{\eta, \d'})_X$. Let $M^\omega_0$ be the direct limit along $(M^n_0: n<\omega)$ and $N^\omega_0$ the direct limit along $(N^n_0: n<\omega)$. For $n<\omega$, let $\k_n$ be the least strong cardinal of $M^n_0$ and $\k_n'$ be the least strong cardinal of $N^n_0$. Let $s_m^n$ be the first $m$ (cardinal) indiscernibles of $L[M^n_0|\k_n]$ and $t_m^n$ be the first $m$ (cardinal) indiscernibles of $L[N^n_0|\k_n']$. Notice that $(M^n_0|\k_n)^\#\in M^n_0$ and $(N^n_0|\k_n')^\#\in N^n_0$. It follows that 
$\tau_{n, l}(s^n_m)=s^l_m$ and $\sigma_{n, l}(t^n_m)=t^l_m$ for $n<l\leq \omega$. We may and do modify the $s^n_m, t^n_m$'s so that $s^\omega_m=u_m$ and $t^\omega_m = u_m$ for each $m$.

We then have the following sequence of implications. Below we let $\Gamma^*$ be the name for the generic extension of $\Gamma$ in the relevant model and $\dot{DM}$ be the name for the derived model. The third implication below uses the fact that $M^1_0 = N^1_0$.
\begin{align*}
W\models \phi[u_m, \Lambda_X, x] & \Rightarrow M^\omega_0[x]\models \emptyset \forces_{Coll(\omega, <\kappa_\omega)}\dot{DM}\models \phi[s^\omega_m, \Gamma^*, x]\\ &\Rightarrow M^1_0[x]=N^1_0[x]\models \emptyset \forces_{Coll(\omega, <\kappa_1)}\dot{DM}\models \phi[s^1_m, \Gamma^*, x] \\ &\Rightarrow N^\omega_0[x]\models \emptyset \forces_{Coll(\omega, <\kappa'_\omega)}\dot{DM}\models \phi[t^\omega_m, \Gamma^*, x] \\ &\Rightarrow W'\models \phi[u_m, \Lambda_X, x].
\end{align*}
The converse has the same proof. We therefore have proved the equivalence and the theorem.

\section{Failure of Tower Sealing}\label{sec:TSfail}

We show in this section that in general, Tower Sealing  fails in hod mice. We recall some terminology from \cite{sargsyan2024exact}. The reader can consult \cite[Definition 2.6]{sargsyan2024exact} for the definition of an \textit{excellent} hybrid premouse $\P$. Let $\P$ be such a premouse and $\P_0\lhd \P$ be the unique lsa type hod premouse that is an initial segment of $\P$ from \cite[Definition 2.6]{sargsyan2024exact}. In particular, there is a Woodin cardinal $\delta_0$ of $\P$ such that $\P_0= (\P|\delta_0)^\sharp\lhd \P$ and $$\P_0\models ``\exists \kappa < \delta_0 \ \kappa \textrm{ is $< \delta_0$-strong and a limit of Woodin cardinals}."$$ Let $\Lambda = S^\P$ be the short-tree strategy predicate for $\P_0$ defined in $\P$, then $\P$ is a $\Lambda$-premouse with a proper class of Woodin cardinals.\footnote{In \cite{sargsyan2024exact}, we demand that $\P$'s class of measurable limit of Woodin cardinals is stationary.} As shown in \cite{sargsyan2024exact}, $\Lambda$ has canonical interpretations in all generic extensions of $\P$; furthermore, let $g\subseteq Coll(\omega,\P_0)$ be $\P$-generic, then $\P[g]\models \textsf{Sealing}$. In fact, Section \ref{sec:PartialTS} shows that $$\P[g]\models \forall \delta \ \textrm{ if } \delta \textrm{ is Woodin, then Partial Tower Sealing holds at } \delta.$$

The next theorem shows that Tower Sealing at any cardinal $\delta$ cannot hold in $\P[g]$.

\begin{theorem}\label{thm:TSfail}
\begin{enumerate}
\item Suppose $(\P,\Sigma)$ is excellent. Let $\P_0,\Lambda$ be the associated lsa type hod premouse derived from $\P$ and $\Lambda$ be the short-tree strategy of $\P_0$ defined in $\P$. Let $g\subseteq Coll(\omega,\P_0)$ be $\P$-generic. Then $$\P[g] \models ``\forall \delta \ \textrm{ if } \delta \textrm{ is Woodin, then Tower Sealing fails at } \delta."$$
\item Suppose $(\P,\Psi)$ is the minimal lbr hod pair such that $\P \models $``there is a strong cardinal and a proper class of Woodin cardinals". Let $\kappa$ be the least strong cardinal of $\P$ and $g\subset Coll(\omega,\kappa^+)$ be $\P$-generic. Then $$\P[g] \models ``\forall \delta \ \textrm{ if } \delta \textrm{ is Woodin, then Tower Sealing fails at } \delta."$$
\end{enumerate}
\end{theorem}

To prove part (1), let $\delta$ be a Woodin cardinal of $\P > \delta_0$, equivalently, $\delta$ is a Woodin cardinal in $\P[g]$. Working in $\P[g]$, let $G\subseteq \mathbb{Q}_{<\delta}$ be $\P[g]$-generic and $j_G:\P[g]\rightarrow M\subseteq \P[g][G]$ be the associated embedding. We show $j_G(\Gamma^g_\infty)\neq \Gamma^{g*G}_\infty.$ Suppose not. Letting $\Gamma = \Gamma^\infty_{g*G} = j_G(\Gamma^\infty_g)$ and $\Lambda_G$ be the canonical interpretation of $\Lambda$ in $\P[g]$, then $$\textrm{Lp}^{\Lambda_G,\Gamma}(\P_0) = \textrm{Lp}^{j_G(\Lambda),\Gamma}(\P_0).$$ $\textrm{Lp}^{\Lambda_G,\Gamma}(\P_0) \unlhd \textrm{Lp}^{j_G(\Lambda),\Gamma}(\P_0)$ since $\Lambda_G\subseteq j_G(\Lambda)$. The other direction holds by our assumption that $\Gamma^\infty_{g*G} = j_G(\Gamma^\infty_g)$; if $\M\lhd \textrm{Lp}^{j_G(\Lambda),\Gamma}(\P_0)$, then letting $\Sigma_\M$ be the unique iteration strategy for $\M$, then $\Sigma_\M\in \Gamma = \Gamma(\P_0,\Lambda_G)$\footnote{The equality $\Gamma = \Gamma(\P_0,\Lambda_G)$ follows from \cite{sargsyan2024exact}, here $\Gamma(\P_0,\Lambda_G)$ is the pointclass generated by $\P_0,\Lambda_G$. $\Gamma(\P_0,\Lambda_G)$ consists of $A\subseteq\mathbb{R}$ such that there is an embedding $i: P_0\rightarrow \Q$ according to $\Lambda_G$ such that $A<_w \Psi$ where $\Psi = (\Lambda_G)_{\Q|i(\kappa)}$.}, so $\M\lhd \textrm{Lp}^{\Lambda_G,\Gamma}(\P_0)$.

Now, let $\M\lhd j_G(\P)$ be the least such that $\rho_\omega(\M)<\omega_1^{\P[g]}$ and $\Sigma_\M$ be the canonical strategy of $\M$ as a $j_G(\Lambda)$-mouse in $M$. Then since $\Sigma_\M$ is a total strategy, it must be in $\Gamma$. This means $\M\lhd \textrm{Lp}^{j_G(\Lambda),\Gamma}(\P_0)$.  $\Sigma_\M$ is universally Baire in $\P[g][G]$, so in $\P[g][G][h]$ where $h$ is $\Coll(\omega,\delta)$-generic over $\P[g][G]$, $\M\lhd \textrm{Lp}^{\Lambda_{G*h}}(\P_0)$ and hence $\M\in \P[g]$ by homogeneity. However, $\M$ defines a surjection from some $\alpha < \omega_1^{\P[g]}$ onto $\omega_1^{\P[g]}$, $\M$ cannot be in $\P[g]$. This is a contradiction. Therefore,  $\Gamma^\infty_{g*G} \neq j_G(\Gamma^\infty_g)$ as claimed.

Now we prove part (2). Let $\kappa,g$ be as in the statement of the theorem. Fix a Woodin cardinal $\delta > \kappa$. Let $G\subseteq\mathbb{Q}_{<\delta}$ be $V[g]$-generic. Let $j_G: V[g]\rightarrow M\subseteq V[g][G]$ be the associated generic embedding. We will show that 
\begin{equation}\label{eqn:TSfails}
j_G(\Gamma^g_\infty)\neq \Gamma^{g*G}_\infty.
\end{equation} 

%Let $(\eta, \delta^*, \delta', \lambda)$ be a good quadruple such that $\eta > \delta$. By results of Section \ref{sec:der_model}, for any $A\in \Gamma^\infty_{g*G}$, there is a club of $X\in D(G,\eta,\delta^*,\lambda)$ such that $X$ captures $A$.

Suppose (\ref{eqn:TSfails}) fails. We write $V$ for the universe of $\P$. Recall that $\iota = \kappa^{+}$ and hence $\iota^+ = \omega_1^{V[g]}$. In $V[g][G]$, let $\M_\infty^G$ be the direct limit of all countable iterates of $\P|\delta$ via $\Psi^{g*G}$ and let $i_G: \P|\delta \rightarrow \M_\infty^G$ be the direct limit embedding. We let $\M_\infty^M$ be the direct limit of all countable iterates of $j_G(\P)|\delta$ via $j_G(\Psi^{g})$ and $i_M: j_G(\P)|\delta\rightarrow \M^M_\infty$ be the direct limit embedding. By our assumption, letting $\Gamma = j_G(\Gamma^g_\infty) =\Gamma^{g*G}_\infty$ and $\Theta = \Theta^\Gamma$, then by the general properties of the direct limit construction, $$i_G(\kappa) = i_M(\kappa) = \Theta$$ and $$\M^G_\infty|\Theta = \M^M_\infty|\Theta = \textrm{HOD}^{L(\Gamma)}|\Theta.$$

In the following, we will write $j$ for $j_G$. Let $(\M,\Phi) \lhd (j(P)|\delta, j(\Psi^g))$ be such that 
\begin{itemize}
\item $(\iota^+)^V < o(\M)$;
\item $\rho_\omega(\M) < (\iota^+)^V$ and $\M$ is the minimal such level of $M$ with this property;
\item $\Phi$ is the canonical strategy of $\M$.
\end{itemize}
Let $\tau = (\iota^+)^V$. We note that $\M|\tau = \P|\tau$ and $\tau$ is a cardinal of $\M$. 

\begin{lemma}\label{lem:plusequal}
$\M^G_\infty|(\Theta)^{+,\M^G_\infty} = \M^M_\infty|(\Theta)^{+,\M^M_\infty}.$
\end{lemma}
\begin{proof}
Work in $V[g][G]$, let $\gamma > > \omega_1$ be a regular cardinal and $X\prec H_\gamma^{V[g][G]}$ be countable such that $\omega_1^{V[g]} < X\cap \delta \in \delta$. Let $\pi_X: M_X\rightarrow X$ be the uncollapse map and $\delta_X = \textrm{crt}(\pi_X)$. Let $(\Q,\R) = \pi_X^{-1}(\P|\delta, j(\P)|\delta)$. Note that $\Q,\R\models ``\kappa$ is strong" and $$\Q|\iota = \R|\iota = \P|\iota = j(\P)|\iota.$$ 
Furthermore, let $G_X = \pi_X^{-1}(G)$, then $\Q = \P|\delta_X$ and $\R = j_{G_X}(\P)|\delta_X$ where $j_{G_X}$ is the generic ultrapower map induced by $G_X$. Since $X$ is transitive below $\delta$, we in fact get (by condensation, cf. \cite[Theorem 4.6]{steel2022condensation}) that $(\Q,\Psi_X) = (\P|\delta_X, \Psi^{g*G}_\Q)$ and $(\R,\Lambda_X) = (j(\P)|\delta_X, j(\Psi^g)_\R)$, where $\Psi_X$ is the $\pi_X^{-1}$-pullback of $\Psi^{g*G}$ and $\Lambda_X$ is the $\pi_X^{-1}$-pullback of $j_G(\Psi^g)$. Note also that $\delta_X >\kappa$ is an inaccessible cardinal in $\P$ and $j(\P)$.

Let $\Theta_X=\pi^{-1}_X(\Theta)$ and $(\Q_\infty,\R_\infty) = \pi_X^{-1}(\M^G_\infty, \M^M_\infty)$. By elementarity, $\Q_\infty|\Theta_X = \R_\infty|\Theta_X$, $\Q_\infty$ is a $\Psi_X$-iterate of $\Q$, and $\R_\infty$ is a $\Lambda_X$-iterate of $\R$. Note that $\Psi_X,\Lambda_X$ are fullness preserving in $L(\Gamma)$. $\Q_\infty|\Theta_X = \R_\infty|\Theta_X$, and by fullness preservation of $\Psi_X,\Lambda_X$\footnote{This is because $\delta_X > \kappa$ is an inaccessible cardinal of $\P$ and $j(\P)$. By condensation, $\Psi_X = \Psi^{g*G}_\Q$, and since $\Psi^{g*G}$ is fullness preserving, $\Psi_X$ is as well. A similar argument applies to $\Lambda_X$.} that $$\Q_\infty|(\Theta_X)^{+,\Q_\infty} = \R_\infty|(\Theta_X)^{+,\R_\infty}.$$ This gives the lemma by elementarity.
\end{proof}

Let $\M_\infty = \M^G_\infty|(\Theta)^{+,\M^G_\infty} = \M^M_\infty|(\Theta)^{+,\M^M_\infty}$. Now let $(\Q,\Sigma)\lhd (\Q^+,\Sigma^+)$ where $(\Q^+,\Sigma^+)\in I(j(\P),j_G(\Psi^g))$\footnote{For a hod pair $(\S,\Upsilon)$, the set $I(\S,\Upsilon)$ denotes the collection of non-dropping iterates $(\S',\Upsilon')$ of $(\S,\Upsilon)$.} and $(\R,\Lambda)\lhd (\R^+,\Lambda^+)\in I(\P,\Psi^{g*G})$ with the following properties:
\begin{enumerate}[(i)]
\item $\M_\infty(\Q,\Sigma)\lhd \M_\infty$ and $\M_\infty(\R,\Lambda)\lhd \M_\infty$. %Furthermore, $o(\M_\infty(\Q,\Sigma))$ and $o(\M_\infty(\R,\Lambda))$ are cardinals of $\M_\infty$.% $> \iota_{\M_\infty}$.
\item $\M_\infty(\Q,\Sigma)\unlhd \M_\infty(\R,\Lambda)$.
\item $\pi^{\Sigma^+}_{\Q^+,\infty}\rest \iota_\Q = \pi^{\Lambda^+}_{\R^+,\infty}\rest \iota_\R$, where $\iota_\Q = \pi^{j_G(\Psi^g)}_{j(\P)|\delta,\Q^+}(\iota)$ and $\iota_\R = \pi^{\Psi^{g*G}}_{\P|\delta,\R^+}(\iota).$
\item $o(\R)$ is a cardinal of $\R^+$ and therefore is a limit of indices of extenders on the sequence of $\R$ with critical point $\kappa_\R$, where $\kappa_\R = \pi^{\Psi^{g*G}}_{\P|\delta,\R^+}(\kappa),$ and $o(\M_\infty(\R,\Lambda))$ is a limit cardinal of $\M_\infty.$
\item $o(\Q)$ is a cardinal of $\Q^+$ and therefore is a limit of indices of extenders on the sequence of $\R$ with critical point $\kappa_\Q$, where $\kappa_\Q = \pi^{j_G(\Psi^g)}_{j(\P)|\delta,\Q^+}(\kappa),$ and $o(\M_\infty(\Q,\Sigma))$ is a limit cardinal of $\M_\infty.$
\end{enumerate}

We note that such pairs can easily be constructed by the general properties of the direct limit systems. Item (ii) follows from the assumption that \ref{eqn:TSfails} fails. Item (iii) follows from the proof of Lemma \ref{lem:plusequal}; indeed, using the notation as there, we can let $\R^+$ be Ult$(\P|\delta,E)$ where $E$ is the (long) extender of length $\Theta_X$ derived from the iteration map $\Pi^{\Psi_X}_{\P|\delta_X,\infty}$ computed in $M_X[G_X]$ and $\Q^+ = \textrm{Ult}(j(\P)|\delta_X,F)$ where $F$ is the (long) extender of length $\Theta_X$ derived from the iteration map $\Pi^{\Lambda_X}_{j_{G_X}(\P)|\delta_X|\delta_X,\infty}$ computed in $M_X[G_X]$, then $$\pi^{\Sigma^+}_{\Q^+,\infty}\rest \iota_\Q = \pi^{\Lambda^+}_{\R^+,\infty}\rest \iota_\R = \pi_X\rest (\Theta_X)^{+,\Q_\infty}.$$ In the following, whenever $(Q,\Psi')\in I(\P|\delta, \Psi^{g*G})$, we write $\kappa_\Q, \iota_\Q$ etc. for the images of $\kappa, \iota$ etc. under the iteration embedding.

Let $\M_Q = \textrm{Ult}_0(\M,K)$ where $K$ is the extender derived from $\pi_{j(\P)|\delta,\Q^+}^{j_G(\Psi^g)}\rest \iota$. Then we have, by standard fine-structural computations, that 
\begin{itemize}
    \item[(vi)] $\rho_1(\M_\Q) < \tau^*_\Q$ where $\tau^*_\Q = \pi_K^M(\tau) = \textrm{sup} \pi_K^M\rest \tau.$
    \item[(vii)] $\M_\Q\lhd \Q^+$. 
\end{itemize}

\begin{lemma}\label{lem:init_equal}
$\pi_{\P|\delta,\R^+}^{\Psi^{g*G}}\rest \iota = \pi_{j(\P)|\delta,\Q^+}^{j_G(\Psi^g)}\rest \iota$ and  $\tau^*_\Q = \pi_{\P|\delta,\R^+}^{\Psi^{g*G}}(\tau).$   

\end{lemma}
\begin{proof}
To see the second clause, first note that the first clause implies that $K$ is the extender derived from $\pi_{\P|\delta,\R^+}^{\Psi^{g*G}}\rest \iota$; furthermore, $\pi_{\P|\delta,\R^+}^{\Psi^{g*G}}$ is continuous at $\tau$ because $\tau$ is a successor cardinal in $\P$ and hence by (vi), $\tau^*_\Q = \pi_{\P|\delta,\R^+}^{\Psi^{g*G}}(\tau).$

To see the first clause, we use the minimality assumption on our mouse $\P$. We note that $$\pi_{\P,\R^+}^{\Psi^{g*G}}\rest \kappa = \pi_{j(\P),\Q^+}^{j_G(\Psi^g)}\rest \kappa.$$ This is because $\Psi^{g*G}_{\P|\kappa} = j(\Psi^g)_{j(\P)|\kappa}$. We write $\sigma$ for this map. Let $T^\P_n$ be the theory of the first $n$-indiscernibles for $\P$. For any $A\subseteq \kappa$, let $\tau$ be a term such that $A = \tau^{\P}[T^\P_n, s]$ for some $s\in [\kappa]^{<\omega}$. Then $\pi_{\P,\R^+}^{\Psi^{g*G}}(A) = \tau^{\R^+}[T^{\R^+}_n, \sigma(s)].$ Similarly, $\pi_{j(\P),\Q^+}^{j(\Psi^{g})}(A) = \tau^{\Q^+}[T^{\Q^+}_n, \sigma(s)].$ By the minimality assumption, $T^{\R^+}_n = T^{\Q^+}_n$ for all $n$. This means $$\pi_{\P|\delta,\R^+}^{\Psi^{g*G}}(A) =\pi_{j(\P)|\delta,\Q^+}^{j(\Psi^{g})}(A)$$ as desired.

\end{proof}

By the choice of $\Q$, (vi), (vii), and Lemma \ref{lem:init_equal}, we easily get that 
\begin{itemize}
    \item[(viii)] $\M_\Q\lhd \Q$.
    \item[(ix)] $\Q|\tau^*_\Q = \R|\tau^*_\Q.$
\end{itemize}

Let $\T$ be the normal tree on $\Q$ according to $\Sigma$ with last model $\M_\infty(\Q,\Sigma)$ and $\U$ the normal tree on $\R$ according to $\Lambda$ with last model $\M_\infty(\R,\Lambda)$.

\begin{lemma}\label{lem:mainbranch}
Suppose $\alpha < \textrm{lh}(\T)$ is on the main branch of $\T$ and $\beta < \textrm{lh}(\U)$ is on the main branch of $\U$. Suppose 
\begin{itemize}
\item $\M^\T_\alpha | \pi^\T_{0,\alpha}(\kappa_\Q) = \M^\U_\beta | \pi^\U_{0,\beta}(\kappa_\R)$,
\item the generators of $[0,\alpha]_T$ are contained in $\pi^\T_{0,\alpha}(\kappa_\Q)$,
\item the generators of $[0,\beta]_U$ are contained in $\pi^\U_{0,\beta}(\kappa_\R)$.
\end{itemize}
Then letting $\kappa^*=\pi^\T_{0,\alpha}(\kappa_\Q)=\pi^\U_{0,\beta}(\kappa_\R)$, $\alpha^*+1$ be the successor $\alpha$ on the main branch of $\T$ (if one exists) and $\beta^*+1$ the successor of $\beta$ on the main branch of $\U$ (if one exists), then if $E^\T_{\alpha^*}$ has critical point $\kappa^*$, then $E^\U_{\beta^*}$ also has critical point $\kappa^*$ and $E^\T_{\alpha^*} = E^\U_{\beta^*}$.
\end{lemma}
\begin{proof}

We show this by induction on the branches $[0,\alpha]_T, [0,\beta]_U$. We assume the lemma holds for pairs $(\alpha',\beta')$ where $\alpha'\in [0,\alpha)_T$, $\beta'\in [0,\beta)_U$. We first make a couple of simple observations. 

First, suppose $\alpha^*+1$ exists, so $E^\T_{\alpha^*}$ is defined. Suppose crt($E^\T_{\alpha^*}) =\kappa^*$. Then $E^\U_{\beta^*}$ is defined and crt($E^\U_{\beta^*})=\kappa^*$. This is easily seen to be true as otherwise, $\kappa_{\M_\infty(\R,\Lambda)} < \kappa_{\M_\infty(\Q,\Sigma)}$, but $\M_\infty(\Q,\Sigma)\unlhd \M_\infty(\R,\Lambda)$. This implies $\M_\infty(\Q,\Sigma)$ has more than one strong cardinal. Contradiction. 

Now observe that if $\kappa^* < \pi^\Lambda_{\R,\infty}(\kappa_\R)$, then $E^\U_{\beta^*}$ exists and crt($E^\U_{\beta^*})=\kappa^*$. This is because if crt$(E^\U_{\beta^*}) > \kappa^*$, then future extenders used along the main branch of $\U$ must have critical points $>\kappa^*$, but this means $\kappa^*$ is a strong cardinal of $\M_\infty(\R,\Lambda)$ since $\alpha$ is on the main branch of $\U$. Since $\pi^\Lambda_{\R,\infty}(\kappa_\R)$ is a strong cardinal of $\M_\infty(\R,\Lambda) > \kappa^*$, this means $\M_\infty(\R,\Lambda)$ has more than one strong cardinal. Contradiction. A similar statement holds for the $\T$-side, namely if $\kappa^* < \pi^\Sigma_{\Q,\infty}(\kappa_\Q)$, then $E^\T_{\alpha^*}$ exists and crt($E^\T_{\alpha^*})=\kappa^*$

The two observations above easily imply most of the conclusions of the lemma except for the last equality. So we assume that $E^\T_{\alpha^*}, E^\U_{\beta^*}$ both exist and have critical point $\kappa^*$. By the initial segment condition, we know that $\lh(E^\T_{\alpha^*})=\lh(E^\U_{\beta^*})=_{def} \xi$. Furthermore, $(\M^\T_{\alpha^*}||\xi,\Sigma_{\M^\T_{\alpha^*}||\xi}) = (\M^\U_{\beta^*}||\xi, \Lambda_{\M^\U_{\beta^*}||\xi})$. Letting $\iota^* = \pi^\T_{0,\alpha}(\iota_\Q)=\pi^\U_{0,\beta}(\iota_\R)$, we have 
\begin{equation}\label{eqn:factorequal}
\pi^{\Sigma^+}_{\M^{\T,+}_\alpha,\infty}\rest \iota^* = \pi^{\Lambda,+}_{\M^{\U,+}_\beta,\infty}\rest \iota^*.
\end{equation}
This follows from (iii) and by our induction hypothesis which implies that $\pi^\T_{0,\alpha}\rest \iota_\Q = \pi^\U_{0,\beta}\rest \iota_\R$. To see (\ref{eqn:factorequal}), let $\xi < \iota^*$, so $\xi = \pi^\T_{0,\alpha}(f)(a) = \pi^\U_{0,\beta}(f)(a)$ for $f\in \Q|\iota_\Q$ and $a\in [\kappa^*]^{<\omega}$. So 
\begin{align*}
\pi^{\Sigma^+}_{\M^{\T,+}_\alpha,\infty}(\xi) & = \pi^{\Sigma^+}_{\Q^+,\infty}(\pi^\T_{0,\alpha}(f))(a) \\
& = \pi^{\Lambda^+}_{\R^+,\infty}(\pi^\U_{0,\beta}(f))(a) \\
& = \pi^{\Lambda^+}_{\M^{\U,+}_\beta,\infty}(\xi). 
\end{align*}
In fact, we get that $\pi^{\Sigma^+}_{\M^{\T,+}_\alpha,\infty}$ and $\pi^{\Lambda,+}_{\M^{\U,+}_\beta,\infty}$ agree on all the elements of the $H_{\iota^*}$ of the models.

Now we can show the equality of the two extenders by the following calculations: let $a\in [\lambda(E^\T_{\alpha^*})]^{<\omega}$ and $A\subseteq [\kappa^*]^{|a|}$,
\begin{align*}
(a,A)\in E^\T_{\alpha^*} & \Leftrightarrow a\in \pi_{E^\T_{\alpha^*}}(A) \\ & \Leftrightarrow a\in \pi^{\Sigma^+}_{\M^{\T,+}_\alpha,\infty}(A) \\ & \Leftrightarrow a\in \pi^{\Lambda^+}_{\M^{\U,+}_\beta,\infty}(A) \\ & \Leftrightarrow (a,A)\in E^\U_{\beta^*}.
\end{align*}
The second equivalence follows from the following facts: 
\begin{itemize}
\item $\pi^\Sigma_{\alpha,\infty}(A) = \pi^\Sigma_{\alpha^*+1,\infty}\circ \pi_{E^\T_{\alpha^*}}(A)$. %Note here that the equality holds because we use Jensen's indexing for our hod mice.
\item By the general properties of direct limits, there is a factor map $\sigma: \M_\infty(\Q,\Sigma)\rightarrow \M_\infty(\Q^+,\Sigma^+)$ such that crt$(\sigma) = \kappa_{\M_\infty(\Q,\Sigma)}$.
\item $\pi^\Sigma_{\alpha^*+1,\infty}(a) = a$. 
\end{itemize}
Combining the above facts, we see that $a\in \pi_{E^\T_{\alpha^*}}(A)$ is equivalent to $a\in \sigma\circ \pi^\Sigma_{\alpha^*+1,\infty}\circ \pi_{E^\T_{\alpha^*}}(A)= \pi^{\Sigma^+}_{\M^{\T,+}_\alpha,\infty}(A).$ This  gives the second equivalence. The third equivalence follows from (\ref{eqn:factorequal}) and the remark after. The last equivalence is proved just like the second equivalence.
\end{proof}

Now we have two cases:

\noindent{\textbf{Case 1:}} $\pi^\T_{0,\infty}(\kappa_\Q)=\pi^\U_{0,\infty}(\kappa_\R)$.

Let this ordinal be $\gamma$. Then $\gamma$ is the strong cardinal of $\M_\infty(\R,\Lambda)$ and of $\M_\infty(\Q,\Sigma)$. Furthermore,  $\pi^{\M_\Q}_F(\M_\Q) \lhd \M_\infty(\R,\Lambda)$ where $F$ is the extender derived from $\pi^\T_{0,\infty}$ and $\pi^{\M_\Q}_F(\M_\Q)$ is the $0$-ultrapower embedding derived from $F$ on $\M_\Q$. Since $\rho_1(\M_\Q)< \tau^*_\Q$, by elementarity, $\rho_1(\pi^{\M_\Q}_F(\M_\Q)) < \pi^\T_{0,\infty}(\tau^*_\Q)$. On the other hand, since $\tau$ is a cardinal of $\P$, $\tau^*_\Q$ is a (successor) cardinal of $\R$ and is a continuity point of $\pi^\U_{0,\infty}$. This means $\pi^\U_{0,\infty}(\tau^*_\Q) = \pi^{\M_\Q}_F(\tau^*_\Q)$ (by Lemma \ref{lem:mainbranch}) is a cardinal of $\M_\infty(\R,\infty)$, but $\pi^{\M_\Q}_F(\M_\Q)$ witnesses $\pi^\U_{0,\infty}(\tau^*_\Q)$ is not a cardinal of $\M_\infty(\R,\infty)$. We have a contradiction.

\noindent{\textbf{Case 2:}} $\pi^\T_{0,\infty}(\kappa_\Q) < \pi^\U_{0,\infty}(\kappa_\R)$.

Let $\alpha$ be the least in $\U$ such that the strong cardinal of $\M^\U_\alpha = \pi^\T_{0,\infty}(\kappa_\Q)$. It's easy to see such an $\alpha$ exists and in fact $\alpha$ in on the main branch of $\U$ (see the analysis in Lemma \ref{lem:mainbranch}). Now we have that letting $E$ be the extender on the main branch of $\U$ that is applied to $\M^\U_\alpha$, then lh$(E)\geq o(\pi^{\M_\Q}_F(\M_\Q))$. This is because $i_E(\pi^\T_{0,\infty}(\kappa_\Q))$ is an inaccessible cardinal of $\M_\infty(\R,\Lambda)$ and by our case hypothesis, $\pi^{\M_\Q}_F(\M_\Q)\lhd \M_\infty(\Q,\Sigma)$, so by the agreement between models, it is easy to see that $\pi^{\M_\Q}_F(\M_\Q)\lhd \M^\U_\alpha$. But this leads to a contradiction as in Case 1 because $\pi^\U_{0,\alpha}(\tau^*_\Q)$ is not a cardinal of $\M^\U_\alpha$.

This completes the proof of Theorem \ref{thm:TSfail}. One weakness of the above proof is it seems very hard to generalize Lemma \ref{lem:plusequal} to obtain the agreements between the two direct limits at a strong cardinal of those limits above $\Theta$. Therefore, one may hope to prove Tower Sealing holds in a generic extension of a hod mouse with two (or more) isolated strong cardinals. We show that this is not the case.

\begin{theorem}[$\sf{AD}^+$]\label{thm:TSfails2}
    Suppose $(\P,\Psi)$ is an lbr hod pair such that $\P \models $``there is a proper class of Woodin cardinals and there are finitely many strong cardinals". Let $\kappa$ be the largest strong cardinal of $\P$, and let $g\subseteq Coll(\omega,\kappa^+)$. Suppose there is no subcompact cardinal in $\P$. Then in $\P[g]$, Tower Sealing fails at every Woodin cardinal.
\end{theorem}
\begin{proof}
    Suppose there are $n$ strong cardinals in $\P$ with $\kappa$ being the largest one. Let $\delta > \kappa$ be a Woodin cardinal. We show Tower Sealing fails at $\delta$ in $\P[g]$. Let $\Q = \P|\delta^{+,\P}$ and fix an $X\prec \Q$ with $|X| = \kappa^+$ and $X\cap \kappa^{++}\in \kappa^{++}$. Let $\pi: M_X\rightarrow X$ be the uncollapse map with crt$(\pi) = \gamma_X$. We may choose $X$ so that $\gamma_X$ does not index an extender on the $\Q$-sequence and cof$(\gamma_X) > \kappa$; there is a $\kappa^+$-club of such $X$ because there are no subcompact cardinals in $\P$. As in \cite{steel2022condensation, normalization_comparison}, we coiterate $\Q$ and the phalanx $(\Q,M_X,\gamma_X)$ into a common hod pair construction.

    More precisely, let $\Sigma = \Psi_{\Q}$ and write $M$ for $M_X$. Fix a coarse strategy pair $((N^*,\in, w, \mathcal{F},\Psi^*),\Psi^{**})$, in the sense of \cite{normalization_comparison}, that captures $\Sigma$, and let $\mathbb{C}$ be the maximal $(w,\mathcal{F})$
construction, 
 with models $M_{\nu,l}$ and induced 
strategies $\Omega_{\nu,l}$. Let $\delta^* = \delta(w)$.
By \cite[Theorem 3.26]{steel2022condensation}, $(*)(M,\Sigma)$ holds,
so we can fix $\la \eta_0,k_0\ra$ lex least such that
  $(\Q,\Sigma)$ iterates to $(M_{\eta_0,k_0},\Omega_{\eta_0,k_0})$, and for 
  all $(\nu,l) <_{\lex}(\eta_0,k_0)$, $(\Q,\Sigma)$ iterates strictly past
  $(M_{\nu,l},\Omega_{\nu,l})$.
Let $\mathcal{U}_{\nu,l}$ be the unique normal tree on $\Q$ witnessing $(\Q,\Sigma)$ 
  iterates past $(M_{\mu,l},\Omega_{\nu,l})$.\footnote{We note that since $k(\Q) = 0$, $\Q$ is strongly stable in the sense of \cite{normalization_comparison}. The possibility that $(\Q,\Sigma)$ iterates to some type 2 pair generated by $(M_{\eta_0,k_0},\Omega_{\eta_0,k_0})$ doesn't occur here.}

    To make the main points  transparent and simplify certain arguments, we assume $n=2$. We write $\kappa_0^\Q < \kappa_1^\Q$ for the strong cardinals of $\Q$ and for any non-dropping iterate $\R$ of $\Q$, we write $\kappa_0^\R,  \kappa_1^\R$ for the strong cardinals of $\R$. Similarly, we denote $\kappa_0^M, \kappa_1^M$ for the strong cardinals of $M$.

  We define trees $\S_{\nu,l}$ on $(\Q,M,\gamma_X)$ for certain $(\nu,l)\leq (\eta_0,k_0)$. 
Fix $(\nu,l)\leq(\eta_0,k_0)$
 for now, and assume $\S_{\nu',l'}$ is defined whenever $(\nu',l')<(\nu,l)$. 
 Let $\U = \U_{\nu,l}$, and for $\tau < \lh(\U)$, let
\[
\Sigma_\tau^\U = \Sigma_{\U \restriction (\tau+1)}
\]
be the tail strategy for $\M_\tau^\U$ induced by $\Sigma$. 
 We proceed to define $ \mathcal{S} = \S_{\nu,l}$,
by comparing the phalanx $(\Q, M,\gamma_X)$ (using strategy $(\Sigma,\Sigma^{\pi_X})$) with $M_{\nu,l}$. As we define
 $\mathcal{S}$, we 
  lift $\mathcal{S}$ to a padded tree $\mathcal{T}$ on $\Q$, by copying. Let us
write
\[
\Sigma_\theta^\T = \Sigma_{\T \restriction (\theta+1)}
\]
for the tail strategy for $\M_\theta^\T$ induced by $\Sigma$.

We let $\Q = \M^\S_0$, $M = \M^\S_1$. For $\theta < lh(\S)$, we will have copy map $\pi_\theta$ from $\M^\S_\theta$ into $\M^\T_\theta$. The map $\pi_\theta$ is a nearly elementary.\footnote{See \cite[Section 2.3]{steel2022condensation} for a summary of the types of elementary maps between mouse pairs.} We attach the complete strategy
\[
 \Lambda_\theta = (\Sigma_\theta^\T)^{\pi_\theta}
\]
  to $\M^\S_\theta$. We also
  define a non-decreasing sequence of ordinals $\lambda_\theta = \lambda^\S_\theta$ 
  that measure
   agreement between models of $\S$, and tell us which 
   model we should apply the next
    extender to.

We start with 
\begin{center}
$\M^\S_0 = \Q, \M_1^\S = M, \gamma_0 = \gamma_X$,
\end{center}
and
\begin{center}
$\M^\T_0 = \M^\T_1 = \Q,  \pi_0 = id,  \pi_1 = \pi_X, \sigma_0 = \pi_X$,
\end{center}
and
\begin{center}
$\Lambda_0 = \Sigma, \mbox{ } \Lambda_1 = \Sigma^{\pi_1}.$
\end{center}

We say that $0, 1$ are distinct roots of $\S$. We say that $0$ is
unstable, and $1$ is stable. As we proceed, we shall declare
additional nodes $\theta$ of $\S$ to be unstable.
We do so because $(\M_\theta^\S,\Lambda_\theta) = 
(\M_\gamma^\U, \Sigma_\gamma^\U)$\footnote{The external
strategy agreement does not seem important to require for $\theta$ to be declared
unstable. We should be able to declare $\theta$ unstable when only the models agree.} for some
$\gamma$, and when we do so, we shall
immediately define $\M_{\theta+1}^\S$, as well as 
$\sigma_\theta$ and $\gamma_\theta$. Here $\Lambda_{\theta+1} =
\Lambda_\theta^{\sigma_\theta}$.
In this case, $[0,\theta]_S$ does not drop, and
all $\xi \le_S \theta$ are also unstable. We regard
$\theta+1$ as a new root of $\S$. This is the only way new roots are
constructed.

If $\theta$ is unstable, then we define
\[
\gamma_\theta = i^\S_{0,\theta}(\gamma_0).%, \ \ \lambda_\theta= \lambda(E^{\M^\S_\theta}_{\gamma_\theta}).
\]

The construction of $\S$ takes place in rounds in which we either
add one stable $\theta$, or one unstable $\theta$ and its
stable successor $\theta+1$. Thus the current last model is always
stable, and all extenders used in $\S$ are taken from stable models.
If $\gamma$ is stable, then $\lambda_\gamma = \lambda(E_\gamma^{\S})$.

For $\theta < lh(\S)$, let $\pi_\theta: \M^\S_\theta\rightarrow \M^\T_\theta$ be the copy map.  We are maintaining by induction that the last node $\gamma$
of our current $\S$  is stable, and

\bigskip
\noindent \textbf{Induction hypotheses $(\dagger)_\gamma$.} If $\theta < \gamma$ 
and $\theta$ is unstable, then 
\begin{itemize}
\item[(1)] $0\leq_\S\theta$ and $[0,\theta]_\S$ does not drop (in model or degree),
and every $\xi \leq_S \theta$ is unstable,
\item[(2)] there is a $\gamma$ such that $(\M_\theta^\S, \Lambda_\theta) = 
(\M_\gamma^\U, \Sigma_\gamma^\U)$,
 \item[(3)] $\M_{\theta +1}^\T = \M_\theta^\T$, and 
$\pi_{\theta+1} = \pi_\theta \circ \sigma_\theta: \M^\S_{\theta+1}\rightarrow \M^\T_{\theta+1} = \M^\T_\theta$.
\item[(4)] $\gamma_\theta$ does not index an extender on the $\M^\S_\theta$-sequence.
\end{itemize}

\medskip

Setting $\sigma_0 = \pi$, we have $(\dagger)_1$.

For a node $\gamma$ of $\S$, we write $\S$-pred$(\gamma)$ for the immediate 
$\leq_\S$-predecessor of $\S$. For $\gamma$ a node in $\S$, we set 
\begin{center}
st$(\gamma) = $ the least stable $\theta$ such that $\theta\leq_\S\gamma$,
\end{center}
and
\begin{center}
st$^*(\gamma)=\begin{cases}
 \text{st$(\gamma)$} & : \ \text{if st$(\gamma)=\theta+1$ for some unstable $\theta$} \\ 
\text{undefined} & : \ \text{otherwise}.
\end{cases}$
\end{center}

The construction of $\S$ ends when we reach a stable $\theta$ 
such that

\begin{enumerate}[(I)]
\item $M_{\nu,l} \lhd \M^\S_\theta$, or $M^\S_\theta = M_{\nu,l}$ and $\S$ drops, or
\item $\M^\S_\theta\unlhd M_{\nu,l}$, 
and $[$rt$(\theta),\theta]_\S$ does not drop in model
 or degree.
% \item There is a stable $\theta$ such that for some $\xi$, $\M^\S_\theta = \M^\U_\xi$ and neither $[rt(\theta),\theta]_S$ nor $[0,\xi]_U$ drops in model or degree. Moreover, letting $Q$ be the result of removing the last extender predicate of $\M^\S_\theta$, we have $Q\lhd M_{\nu,l}$.
\end{enumerate}
If case (I) occurs, then we go on to define $\S_{\nu,l+1}$. If case (II) occurs, 
we stop the construction.

We now describe how to extend $\S$ one more step. First we assume $\S$ has successor
 length $\gamma+1$ and let $\M^\S_\gamma$ be the current last model, so that $\gamma$ is
  stable. Suppose $(\dagger)_\gamma$ holds. Suppose (I), (II) above do not hold 
  for $\gamma$, so that we have a least
   disagreement between $\M^\S_\gamma$ and $M_{\nu,l}$.  Suppose the least
disagreement involves only an extender $E$ on the sequence of $\M^\S_\gamma$.\footnote{Later, we will prove that this is the case.} 
Letting $\tau = \lh(E)$, we have
\begin{itemize}
\item $M_{\nu,l}|(\tau,0) = \M^\S_\gamma|(\tau,-1)$,\footnote{Recall 
$\M^\S_\gamma|(\tau,-1)$ is the structure obtained from $\M^\S_\gamma|\tau$
 by removing $E$. Sometimes, we will write $M^-$ for $M|(o(M),-1)$.} and
\item ${(\Omega_{\nu,l})}_{(\tau,0)}=(\Lambda_\gamma)_{(\tau,-1)}$.
\end{itemize}

We now describe how to extend $\S$ one more step. We set $E^\S_\gamma = E^+$ and $\lambda^\S_\gamma = \lambda_E$.\footnote{This is the notation used in \cite{normalization_comparison}, for an extender $E$ on the $M$-sequence, $E^+$ is the extender with generators $\lambda_E\cup \{\lambda_E\}$ that represents 
$i_F^{\textrm{Ult}(M,E)}\circ i^M_E$, where $F$ is the order zero total measure on $\lambda_E$ in Ult$(M,E)$. We also write $\hat{\lambda}(E^+) = \lambda_E$,  lh$(E^+) = \textrm{lh}(E)$. $E^+$ is the plus-type extender derived from $E$.} Let $\xi$ be the least such that 
crt$(E)<\lambda^\S_\xi$. We let $\S$-pred$(\gamma+1)=\xi$. Let $(\beta,k)$ be 
lex least such that either $\rho(\M^\S_\xi|(\beta,k))\leq $crt$(E)$ or 
$(\beta,k)=(\hat{o}(\M^\S_\xi),k(\M^\S_\xi))$. Set
\begin{center}
$\M^\S_{\gamma+1}=$Ult$(\M^\S_\xi|(\beta,k),E^+)$,
\end{center}
and let $\hat{i}^\S_{\xi,\gamma+1}$ be the canonical embedding. Let 
\begin{center}
$\M^\T_{\gamma+1}=$Ult$(\M^\T_\xi|(\pi_\xi(\beta),k),\pi_\gamma(E)^+)$,
\end{center}
and let $\pi_{\gamma+1}$ be given by the Shift Lemma. This determines
$\Lambda_{\gamma +1}$.%\footnote{See \cite{steel2025JSZ} for more details on how the copying process works in the case $\gamma$ is special and (3)(b)(ii) of Definition \ref{dfn:special} holds.}

If $\xi$ is stable or
 $(\beta,k)< (\hat{o}(\M^\S_\xi),k(\M^\S_\xi))$, then we declare $\gamma+1$ to be stable.
 $(\dagger)_{\gamma+1}$ follows vacuously from $(\dagger)_\gamma$.

If $\xi$ is unstable and $E^+$
is not used in $\U$, then again we declare $\gamma+1$ stable. 
Again, $(\dagger)_{\gamma+1}$ follows vacuously from $(\dagger)_\gamma$.

 Finally, suppose 
 $\xi$ is unstable and $E^+$ is used in $\U$, say $E^+ = E^\U_\mu$.
 Let $\tau$ be such that \[e^\S_{\xi} = e^\U_\tau,\] where $e^\S_{\xi}$ is the sequence of extenders used on the branch $[0,\xi]_S$ and similarly for $e^\U_\tau$. So in particular,
\[
(\M^\S_{\xi},\Lambda_{\xi}) = (\M_\tau^\U, \Sigma_\tau^\U).
\]
 We have that 
 \[
 e^\S_{\gamma+1} = {e^\S_\xi}^\smallfrown \langle E^+ \rangle = {e^\U_\tau}^\smallfrown \langle E^+ \rangle = e^\U_{\mu+1}.
 \]
 \cite{steel2025JSZ} shows that $\tau = U-pred(\mu+1)$. We then we declare $\gamma+1$ to 
 be unstable
  and $\gamma+2$ stable. 
  We must define the tuple needed for $(\dagger)_{\gamma+2}$.
  Let $i = i_{\xi,\gamma+1}^\S$, and
  \begin{center}
  $\la N, G,\sigma,\gamma^* \ra = \la \M_\xi^\S,
\M_{\xi+1}^\S, \sigma_\xi,\gamma_\xi \ra.$
  \end{center}
  
We let
      
  \begin{center}
$\gamma_{\gamma+1} =  i(\gamma^*) = i^\S_{0,\gamma+1}(\gamma_0).$
\end{center}
Now we define $\M^\S_{\gamma+2}$ and $\sigma_{\gamma+1}$. Note that by our assumption that cof$(\gamma_X)>\kappa$ and the fact that crt$(E)\leq \kappa^{\M^\U_\tau}$, where $\kappa^{\M^\U_\tau}$ is the largest strong cardinal of $\M^\U_\tau$, $$\gamma_{\gamma+1} = \textrm{sup} \  i[\gamma^*].$$ Let $\M^\S_{\gamma+2} = \textrm{Ult}(G, E^+)$, $i^\S_{\xi+1,\gamma+2}$ be the ultrapower map, and $\sigma_{\gamma+1}: \M^\S_{\gamma+2} \rightarrow \M^\S_{\gamma+1}$ be the copy map and $\pi_{\gamma+1} = \pi_\gamma\circ \sigma_{\gamma+1}$.

If there is a least disagreement between $\M^\S_{\gamma+2}$ and $M_{\nu,l}$, it has to involve an extender $F$ from the sequence of $M^\S_{\gamma+2}$ (by \cite[Lemma 5.64]{normalization_comparison}). If no such $F$ exists, we leave $\lambda^\S_{\gamma+1},\lambda^\S_{\gamma+2}$ undefined. Otherwise, let 
\begin{center}
$\lambda^\S_{\gamma+2} = \lambda(F)$
\end{center}
and
\begin{center}
$\lambda^\S_{\gamma+1} = \textrm{min}(\lambda^\S_{\gamma+2}, \gamma_{\gamma+1})$.
\end{center}
The $\lambda^\S_\xi$'s tell us what model should an extender used in $\S$ be applied to.

\begin{claim}\label{claim:dag}
    $(\dag)_{\gamma+1}$ holds.
\end{claim}
\begin{proof}
    (1)--(3) are clear. (4) also follows because $\gamma_{\gamma+1} = i^\S_{0,\gamma+1}(\gamma_0)$. By the fact that $\gamma_0$ does not index an extender on the $\M^\S_0$-sequence and elementarity, $\gamma_{\gamma+1}$ does not index an extender on the $\M^\S_{\gamma+1}$-sequence.
\end{proof}

 If (I) or (II) holds at $\gamma+2$, then the construction of $\S$
            is over. Otherwise, we let $E_{\gamma+2}^{\S}$ be the least disagreement
            between $\M_{\gamma+2}^{\S}$ and $M_{\nu,l}$, and we set 
            \[
            \lambda^\S_{\gamma+1} = \inf(\gamma_{\gamma+1},\lambda(E_{\gamma+2}^{\S})).
            \]
            This completes the successor step in the construction of $\S$.

Now suppose we are given $\S \restriction \theta$, where
$\theta$ is a limit ordinal. Let $b=\Sigma(\T\rest\theta)$.
\\
\\
\noindent \textbf{Case 1.} There is a largest $\eta\in b$ such that $\eta$ is unstable.

\medskip

Fix $\eta$. There are two subcases.
\begin{enumerate}[(A)]
\item for all $\gamma\in b-(\eta+1)$, rt$(\gamma)=\eta+1$.
 In this case, $b-(\eta+1)$ is a branch of $\S$. Let $\S$ choose this branch,
\begin{center}
$[\eta+1,\theta)_\S=b-(\eta+1)$,
\end{center}
and let $\M^\S_\theta$ be the direct limit of the $\M^\S_\gamma$ for sufficiently 
large $\gamma\in b-(\eta+1)$. We define 
the branch embedding $i^\S_{\gamma,\theta}$ a usual and 
$\pi_\theta:\M^\S_\theta\rightarrow \M^\T_\theta$ is given by 
the fact that the copy maps commute with the branch embeddings.
 We declare $\theta$ to be stable.

\item for all $\gamma\in b-(\eta+1)$, rt$(\gamma)=\eta$. Let $\S$ choose
\begin{center}
$[0,\theta)_\S = (b-\eta)\cup[0,\eta]_\S$,
\end{center}
and let $\M^\S_\theta$ be the direct limit of the $\M^\S_\gamma$ for 
sufficiently large $\gamma\in b$.
 Branch embeddings $i^\S_{\gamma,\theta}$ for $\gamma\geq \eta$ are defined as usual.
  $\pi_\theta:\M^\S_\theta\rightarrow \M^\T_\theta$ is given by the fact that copy
   maps commute 
  with branch embeddings. We declare $\theta$ to be stable.
\end{enumerate}
Since $\theta$ is stable,  $(\dagger)_\theta$ follows at once from
$\forall \gamma < \theta  \ (\dagger)_\gamma$.

\noindent \textbf{Case 2.} There are boundedly many unstable ordinals in $b$ but no largest one.

We let $\eta$ be the sup of the unstable ordinals in $b$. Let $\S$ choose 
\begin{center}
$[0,\theta)_\S = (b-\eta)\cup[0,\eta]_\S$,
\end{center}
and define the corresponding objects as in case 1(B). We declare
 $\theta$ stable, and again $(\dagger)_\theta$ is immediate.

\noindent \textbf{Case 3.} There are arbitrarily large unstable ordinals in $b$. 
In this case, $b$ is a disjoint union of pairs $\{\gamma,\gamma+1\}$ such that
 $\gamma$ is unstable and $\gamma+1$ is stable. We set
\begin{center}
$[0,\theta)_\S=\{\xi\in b \ | \ \xi \textrm{ is unstable}\},$
\end{center}
and let $\M^\S_\theta$ be the direct limit of the $\M^\S_\xi$'s for $\xi\in b$ unstable.
 There is no dropping in model or degree along $[0,\theta)_\S$. We define maps
  $i^\S_{\xi,\theta}, \pi_\theta$ as usual. If $(\M^\S_\theta, \Lambda_\theta)$ is not a pair of 
the form $(\M_\tau^\U, \Sigma_\tau^\U)$, then
 we declare $\theta$ stable and $(\dagger)_\theta$ is immediate. 
 
     Suppose that $(\M_\theta^\S, \Lambda_\theta)$ is a pair of $\U$. We declare
$\theta$ unstable. We set
\[
\gamma_\theta =  i_{0,\theta}^\S(\gamma_0)
\]
and
\[
\M_{\theta+1}^\S= \text{ the direct limit of $i_{\gamma+1,\gamma'+1}^\S(\M_{\gamma+1}^\S)$, for $\gamma <_S \gamma' <_S 
\theta$.}
\]
We also let
\[
\sigma_\theta= \text{ common value of $i_{\gamma,\theta}^\S(\sigma_\gamma)$, for $\gamma <_S 
\theta$ sufficiently large.}\footnote{We abuse the notation a bit here when we write $i^\S_{\gamma,\theta}(\sigma_\gamma)$ as $\sigma_\gamma$ is technically not an element of $\M^\S_\gamma$, but the meaning of $\sigma_\theta$ should be clear.}
\]
It is easy then to see that
$\Phi_\theta = \la \M_\theta^\S, \M_{\theta+1}^\S,
 \sigma_\theta, \gamma_\theta \ra$ witnesses $(\dag)_\theta$ holds.

If (I) holds, then we stop the construction of $\S = \S_{\nu,l}$ and move on
to $\S_{\nu,l+1}$. If (II) holds, we stop the construction of $\S$ and do not
move on. If neither holds, we let $E_{\theta+1}^\S$ be the extender on
the $\M_{\theta+1}^\S$ sequence that represents its first disagreement
with $M_{\nu,l}$, and set
\[
\lambda^\S_{\theta+1} = \lambda(E_{\theta+1}^\S),
\]
\[
\lambda^\S_\theta = \inf(\lambda_{\theta+1}, \gamma_\theta).
\]
It then is routine to verify $(\dagger)_{\theta+1}$.

This finishes our construction of $\S=\S_{\nu,l}$ and $\T$. Note that every extender used 
in $\S$ is taken from a stable node and every stable node, except the last model of $\S$ 
contributes exactly one extender to $\S$. The last model of $\S$ is stable.

\begin{remark}
    It is possible in general
  that  $\xi$ is unstable, $\S$-pred$(\gamma+1)=\xi$, and 
  crt$(E^\S_\gamma)=\lambda_F$ where $F$ is the last extender of $\M^\S_\xi|\gamma_\xi$. 
  In this case, $(\beta,k) = (\lh(F),0)$. The problem then is that
   $\M^\S_{\gamma+1}$ is not an lpm, because its last extender
   $i_{\xi,\gamma +1}(F)$ has a missing whole initial segment,
   namely $F$. This is the \textit{JSZ anomaly}. 

   In the situations of least-disagreement comparisons, when a JSZ extender occurs, the Schindler-Zeman solution is to just continue comparing anyway. Suppose a JSZ extender, which fails the Jensen ISC, has the form $ E^\S_{\eta+1}= L\circ E$ is used on the $\mathcal{U}$-side. Say our phalanx is $(M,Q,\gamma_0)$ where $E^\mathcal{S}_\eta = L$ with crt$(L)=\lambda_E$ and $E = E^M_{\gamma_0}$, then $E\notin M^\S_{\eta+1} = \textrm{Ult}(M|\gamma_0,L)$ and $\M^\S_{\eta+2} = \textrm{Ult}(M,E^\S_{\eta+1}) = \textrm{Ult}(M,L\circ U)$.  Note that the Jensen ISC holds everywhere on $\U$, then the main branch of $\mathcal{U}$ uses first $E$ and then $L$. So $L$ is used on both $\S$ and $\U$; this cannot happen in a comparison by least-extender disagreement. Since our comparison is against a common background construction, the SZ solution does not seem to work here; some care must be taken when JSZ anomalies occur. The JSZ anomaly affects how we lift our problematic phalanx and forces us to modify the rules of $\mathcal{S}_{\nu,l}$ to enable to prove the comparison terminates and other aspects of the comparison. See \cite{steel2025JSZ} for how to handle the JSZ anomalies in the context of comparison against a background construction.

   The JSZ anomaly does not occur in the comparison we are describing in this paper. The reason is that we chose $X$ so that $\gamma_X$ does not index an extender on the $\Q$-sequence and this fact propagates to the lifted phalanxes.
\end{remark}

\begin{claim}\label{claim:termination}
For some $(\nu,l)\leq (\eta_0,k_0)$, the construction of $\S_{\nu,l}$ stops for reason (II).
\end{claim}
\begin{proof}
This is similar to the proof of \cite[Lemma 9.6.2]{normalization_comparison}.

\end{proof}

Fix  $(\nu,l)\leq (\eta_0,k_0)$ such that the construction of $\S=\S_{\nu,l}$ 
terminates at a stable $\theta$ such that for some $\gamma$, $\M^\S_\theta \unlhd \M^{\U_{\nu,l}}_\gamma$.  Let $\S=\S_{\nu,l}$, 
$\U = \U_{\nu,l}$,
and let $\gamma$ be the least such that $\M^\S_\theta\unlhd \M^\U_\gamma$. 
We have lh$(\S)=\theta+1$, and $[\rm{rt}$$(\theta),\theta]_\S$ does not drop in model or degree.

\begin{claim}\label{claim:unstable}
For some unstable $\xi$, rt$(\theta) = \xi+1$.
\end{claim}
\begin{proof}
    Suppose the claim is false. Then
\begin{enumerate}
\item[(i)] either rt$(\theta) = \xi \geq 0$ for $\xi$ unstable,
\item[(ii)] or rt$(\theta) = \xi$ for $\xi$ stable and is a limit of unstable $\xi' < \xi$,
%\item the first extender used along the branch $[0,\theta]_S$ is $F$, where $F$ is the top extender of some $\M^\S_\xi$, $1\leq_S \xi$ and $[1,\xi]_S$ does not drop,
%\item $\M^\S_\theta = \M^\U_\gamma$.
\end{enumerate}

In either case, the embeddings $i^\S_{0,\xi}$ and $i^\S_{\xi,\theta}$ exist, so $i^\S_{0,\theta}$ exists. Let $P = \M^\S_\theta$ and $R = \M^\U_\gamma$. Let $i = i^\S_{0,\theta}$, $j = i^\U_{0,\gamma}$, $\sigma: P\rightarrow P^* = \M^\T_\theta$ be the copy map, and $i^*: M\rightarrow P^*$ be the branch embedding of $[0,\theta]_T$ if these maps are defined.

By Dodd-Jensen, $P= R$, $i,j, i^*$ are defined, and $i= j$. Let $H$ be the first extender used along $[0,\theta]_S$ and $K$ be the first extender used along $[0,\gamma]_U$. Since $i=j$, $K, H$ are compatible. 

%Since $H$ may fail the Jensen ISC and $K$ satisfies the Jensen ISC, $K$ must be an initial segment of $H$, i.e. $K\rest \lambda_K = H\rest \lambda_K$ and $\lambda_K < \lambda_H$ if $H$ fails the Jensen ISC.

Since we can recover branch extenders from branch embeddings, we have $$e^\S_\theta = e^\U_\gamma.$$ Let $\eta \leq_S \theta$ be the least stable. Then $e^\S_\eta = e^\S_\theta\rest \delta = e^\U_\gamma \rest \delta$ for some $\delta$.\footnote{If rt$(\theta)=0$ then $e^\S_\eta$ consists of a single extender $H$ and $\delta =1$. This is a special case and is simpler.} There is a $\tau\leq_U \gamma$ such that $e^\U_\tau = e^\U_\gamma \rest \delta = e^\S_\eta$. So $\M^\S_\eta = \M^\U_\tau$. By pullback consistency,\footnote{Pullback consistency follows from other properties of mouse pairs specified in \cite{normalization_comparison}.} we easily get that $\Lambda_\eta = \Sigma^\U_\tau$. If $\eta$ is a limit ordinal, then by the rule of $\S$, we declare $\eta$ unstable, contradicting our assumption. So let $S-pred(\eta) = \mu$; then $\mu$ is unstable by the minimality of $\eta$. But then we declare $\eta$ unstable by the rule of $\S$ at successor stages. Again this is a contradiction.
\end{proof}

%\footnote{Here we use the fact that the models have $\lambda$-indexing. In particular, $i_E(\pi^\T_{0,\infty}(\kappa_\Q)$ }

%Note that the branch embedding $[\alpha, lh(\U)-1]$ has critical point $> \pi^\U_{0,\alpha}(\tau_\R)$.  

Let $\xi$ be as in Claim \ref{claim:unstable}, and $\tau$ be such that $(\M^\S_\xi,\Lambda_\xi) = (\M^\U_\tau,\Sigma^\U_\tau)$. We have that 
\begin{itemize}
\item $(\M^\S_\theta,\Lambda_\theta) \unlhd (M_{\nu,l},\Omega_{\nu,l}) \unlhd (\M^\U_\gamma, \Sigma^\U_\gamma)$ for some $(\nu,l)\leq (\eta_0,k_0)$.
\item $[\xi+1,\theta]_S$ does not drop in model or degree.
\item The tuple $(\M^\S_\xi,\M^\S_{\xi+1},\sigma_\xi, \gamma_\xi)$ witnesses $(\dag)_\xi$.
\end{itemize}
Let $P = \M^\S_\theta$ and $R = \M^\U_\gamma$. See Figure \ref{Comparison} for the relevant diagram of the comparison.
%Let $\eta$ be the least $\gamma\leq_S \theta$ such that either $\gamma = \theta$ or crt$(i^\S_{\gamma,\theta})>i^\S_{\xi+1,\gamma}(\gamma_\xi)$ and let $\gamma^* = i^\S_{\xi+1,\gamma}(\gamma_\xi)$. 
%\cite[Claim 10]{steel2025JSZ} shows that either $\eta = \xi+1$ or else $\eta > \xi+1$ and $\eta$ is $\U$-accessible.

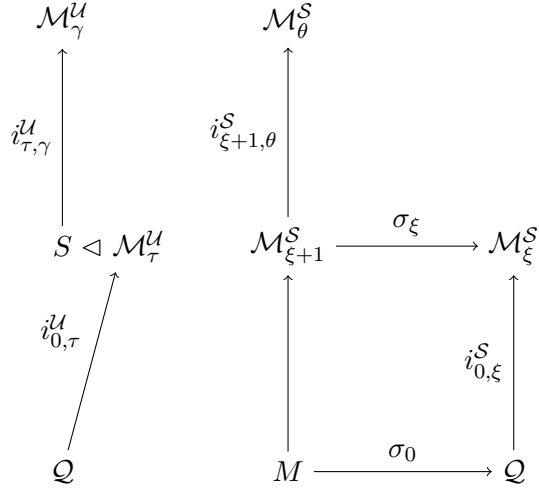
\begin{figure}
\centering
\begin{tikzpicture}[node distance=3cm, auto]
  \node (A) {$\Q$};
  \node (B) [above of=A] {$S$};
  \node (C) [node distance=0.8cm, right of=B] {$\lhd \ \M^\U_\tau$};
  \node (D) [node distance =3cm,above of=B] {$\M^\U_\gamma$};
  %\node (E) [node distance=2cm, above of=C] {$Q_\tau, F$};
  %\node (F) [node distance=6.5cm,above of=D] {$Q_{j(\kappa)}^{j(\U)}$};
  %\node (Z) [node distance=1.9cm, above of=B] {};
  
  \draw[->] (A) to node {$i^\U_{0,\tau}$} (C);
  \draw[->] (B) to node {$i^\U_{\tau,\gamma}$} (D);
  %\draw[->] (B) to node  {$F$} (Z);
  %\draw[->] (D) to node {$i^{j(\U)}_{\tau+1,j(\kappa)}$} (F);
  %\draw[->] (Z) to node  {$D$} (D);

  \node (G) [right of=A]{$M$};
  \node (H) [above of=G]{$\M^\S_{\xi+1}$};
  \node (I) [right of=G]{$\Q$};
  \node (J) [above of=I]{$\M^\S_{\xi}$};
  \node (K) [node distance= 3cm, above of=H]{$\M^\S_\theta$};
  %\node (L) [node distance=5cm, above of=J]{$P_{\mu+1}, H$};
  %\node (M) [node distance=1.5cm, above of=L]{$P^{j(\S)}_{\mu+2}$};
  %\node (N) [node distance=2cm, above of=M]{$P^{j(\S)}_\theta, G$};
  %\node (P) [node distance = 6.8cm,above of=H]{$P^{j(\S)}_{\theta+1}$};
  %\node (Q) [node distance =3.1cm, above of=P]{$P_{j(\kappa)}^{j(\S)}$};
  %\node (R) [node distance=2.9cm, right of=G]{Q};
  %\node (S) [node distance=2.9cm, right of=H]{$P_{\kappa+1}$};

  \draw[->] (G) to node {$\sigma_0$} (I);
  \draw[->] (H) to node {$\sigma_\xi$} (J);
  \draw[->] (G) to node {} (H);
  \draw[->] (I) to node {$i^{\S}_{0,\xi}$} (J);
  %\draw[->] (G) to node {$i^{\S}_{0,\mu}$} (H);
  %\draw[->] (J) to node {} (L);
  \draw[->] (H) to node {$i^\S_{\xi+1,\theta}$} (K);
  %\draw[->] (M) to node {$i^{j(\S)}_{\mu+2,\theta}$} (N);
  %\draw[->] (R) to node {$\sigma_0$} (G);
  %\draw[->] (R) to node {} (S);
  %\draw[->] (L) to node {$\sigma_\mu$} (K);
  %\draw[->] (S) to node {} (L);
  %\draw[->] (S) to node {$\sigma_\kappa$} (H);
  
\end{tikzpicture}
\caption{Diagram of the comparison argument: $\M^\S_\theta\unlhd \M^\U_\gamma$, $\M^\S_\xi = \M^\U_\tau$, crt$(i^\S_{\xi+1,\theta})>\gamma_\xi = \crt(\sigma_\xi)$, and the embedding $i^\U_{\tau,\gamma}$ acts on $S\lhd \M^\U_\tau$ where $S$ the the collapsing structure of $\gamma_\xi$.}
\label{Comparison}
\end{figure}

\begin{claim}\label{claim:no_drop}
\begin{enumerate}
\item[(i)] $\tau \leq \eta <\gamma$ implies lh$(E^\U_\eta)\geq \gamma_\xi$ and if $\eta < \tau$ then $\lambda(E^\U_\eta)\leq \gamma_\xi$.
\item[(ii)] $P \lhd R$. %and $[0,\gamma]_U$ does not drop in model or degree.
\end{enumerate}
\end{claim}
\begin{proof}
    The first clause of part (i) follows from the agreement between $P, M_{\nu,l}, R, \M^\S_\xi = \M^\U_\tau$, more precisely, these models agree up to $\gamma_\xi$, therefore, for any $\tau\leq \eta < \gamma$, lh$(E^\U_\eta)\geq \gamma_\xi$. For the second clause, note that by construction, $e^\U_\tau = e^\S_\xi$\footnote{$e^\U_\tau$ is the extender sequence used along $[0,\tau]_U$ and similarly $e^\S_\xi$ is the extender sequence used along $[0,\xi]_S$.} and lh$(e^\U_\tau) = \lh(e^\S_\xi)\leq \gamma_\xi$; the fact that $\lh(e^\S_\xi)\leq \gamma_\xi$ follows from the rules of lifting phalanx, all extenders used in $e^\S_\xi$ have critical point less than $\gamma_\xi$ and therefore, their length has to be $\leq \gamma_\xi$. If $\eta < \tau$ then $\lambda(E^\U_\eta) \leq  \lh(E^\U_\eta)\leq \lh(e^\U_\tau)\leq \gamma_\xi$ as desired.

    To see part (ii), suppose $P= R$. First note that the branch $[0,\gamma]_U$ cannot drop because $P$ is a $\sf{ZFC}^-$-model while if $[0,\gamma]_U$ drops then $R$ is not. So $[0,\gamma]_U$ has no drop. We may assume $\gamma > \tau$ as otherwise, $P=R=\M^\U_\tau$ but $\M^\U_\tau$ has a collapsing level for $\gamma_\xi$ while $\gamma_\xi$ is a cardinal in $P$. Contradiction. So $\gamma>\tau$. Let $\tau'\in [0,\gamma]_U$ be the least element of this branch $\geq \tau$. It is easy to see that $\M^\U_\tau$ and $\M^\U_{\tau'}$ agree up to their common $\gamma_\xi^+$ and the branch embedding $i^\U_{\tau',\gamma}$ has critical point $>\gamma_\xi$. But this is also a contradiction because $\M^\U_\tau$, hence $\M^\U_{\tau'}$ and $R$, has a collapsing level for $\gamma_\xi$ while $\gamma_\xi$ is a cardinal in $P$.
\end{proof}

\begin{claim}\label{claim:taumainbranch}
$\tau\in [0,\gamma]_U$ and letting $\epsilon+1\in [\tau,\gamma]_U$ be the $U$-successor of $\tau$ on the branch $[0,\gamma]_U$, crt$(E^\U_\epsilon)>\gamma_\xi$.
\end{claim}
\begin{proof}
    Let $\alpha <_U \beta+1 \in [0,\gamma]_U$ be such that $\alpha \leq \tau$, $\beta+1 > \tau$, and $\alpha$ is the $U$-predecessor of $\beta+1$.  Suppose crt$(E^\U_\beta)<\gamma_\xi$, then we claim that
 \\
 \\
\noindent \textbf{Subclaim.}\label{claim:strong}
    crt$(E^\U_\beta)$ must be a strong cardinal in $\M^\U_\a$ and $i^\U_{\alpha,\beta+1}$ exists.%\footnote{\label{footnote:finite}The proof is by induction on $\beta$ and uses the fact that the set of strong cardinals in the models of $\U$ is finite.} 
\begin{proof}
First note that by an easy induction on $\beta\geq \tau$, crt$(E^\U_\beta)$ must be a strong cardinal in $\M^\U_\tau$. Second, note that since $\alpha\leq \tau$, for $i\in \{0,1\}$, $\kappa_i^{\M^\U_\alpha}\leq \kappa_i^{\M^\U_\tau}$.  \footnote{Note is that it can't happen that crt$(E^\U_\beta)=\kappa_0^{\M^\U_\tau}$ and $\kappa_0^{\M^\U_\tau}<\kappa_0^{\M^\U_\alpha}<\kappa_1^{\M^\U_\alpha}<\kappa_1^{\M^\U_\tau}$. This is because this means $\lambda(E^\U_{\alpha-1})>\kappa_0^{\M^\U_\tau}$, therefore, since crt$(E^\U_\beta)=\kappa_0^{\M^\U_\tau}$, the rule of $\U$ implies that $E^\U_\beta$ must be applied to model earlier than $\alpha$. Contradiction.}

So let us assume $\kappa = \crt(E^\U_\beta) = \kappa_i^{\M^\U_\tau}$ for some $i$ and $ \kappa_i^{\M^\U_\alpha} <  \kappa_i^{\M^\U_\tau}$.  Since $\M^\U_{\beta+1}|\gamma_\xi = P|\gamma_\xi$, in $P$, $\kappa_i^{\M^\U_\alpha}$ is strong to $\kappa_i^{\M^\U_\tau}$; since $\kappa_i^{\M^\U_\tau}$ is strong in $P$, $\kappa_i^{\M^\U_\alpha}$ is strong in $P$ as well.  We get that $$P\models ``\textrm{ there are } n+1 \textrm{ strong cardinals}."$$ This is a contradiction. So $ \kappa_i^{\M^\U_\alpha} =  \kappa_i^{\M^\U_\tau}$ as desired. Furthermore, this easily gives that $\M^\U_\alpha$ and $\M^\U_\beta$ agrees up to the successor cardinal of $\kappa^{\M^\U_\alpha}$. This gives the second clause and completes the proof of the subclaim.

\end{proof}
    Using the subclaim, let $\kappa = \crt(E^\U_\beta) = \kappa_i^{\M^\U_\alpha}$ for some $i$, and $\lambda = i^\U_{\alpha,\beta+1}(\kappa)$. Then  since lh$(E^\U_\beta) > \gamma_\xi$, $\lambda>\gamma_\xi$ and $\lambda$ is a strong cardinal in $\M^\U_{\beta+1}$. Furthermore, $\lambda  < \lh(E^\U_\beta) \leq o(P)$ and $\M^\U_{\beta+1}||\lambda = P||\lambda$. Since $i^\S_{\xi+1,\theta}$ is above $\gamma_\xi$, the strong cardinals of $P$ are below $\gamma_\xi$. In particular, $$\M^\U_{\beta+1}||\lambda = P||\lambda\models ``\textrm{ there are $n$ strong cardinals.}"$$ But then since $\lambda$ is strong in $\M^\U_{\beta+1}$, $$\M^\U_{\beta+1}\models ``\textrm{ there are $n+1$ strong cardinals}".$$ This is a contradiction.
    
    We have shown that letting $\alpha, \beta+1$ be as above, then crt$(E^\U_\beta)\geq\gamma_\xi$, in fact it is easy to see that crt$(E^\U_\beta) > \gamma_\xi$. If $\tau\notin [0,\gamma]_U$, then $\alpha < \tau$ and hence $\lambda(E^\U_\alpha) \leq \gamma_\xi$, but then crt$(E^\U_\beta) < \lambda(E^\U_\alpha) \leq \gamma_\xi$. Contradiction. This shows $\tau\in [0,\gamma]_U$ and completes the proof of the claim.

\end{proof}

Claim \ref{claim:taumainbranch} and the argument in Claim \ref{claim:no_drop} imply that the branch $[\tau,\gamma]_U$ must drop, in fact, letting $S\lhd \M^\U_\tau$ be the collapsing structure for $\gamma_\xi$, $\M^\U_{\epsilon+1}=\Ult(S,E^\U_\epsilon)$. In other words, the branch $[\tau,\gamma]_U$ is based on $S$. Let $i: S\rightarrow R$ be the iteration embedding and $j: \M^\S_{\xi+1}\rightarrow P$ be the iteration embedding $i^\S_{\xi+1,\theta}$. We have that by pullback consistency, $(\Sigma^\U_{\gamma})^i = (\Sigma^\U_\tau)_S$ and $\Lambda_\theta^j = \Lambda_{\xi+1}$. Claims \ref{claim:no_drop} and \ref{claim:taumainbranch} easily imply that $\Lambda_{\xi+1}$ is projective in $(\Sigma^\U_\tau)_S$. Similarly, letting $S^*\lhd \Q$ be the collapsing structure for $\gamma_0$, $\Lambda_1$ is projective in $(\Sigma^\U_0)_{S^*}$.

The above gives us the following: if $j: \P[g]\rightarrow M$ is a generic ultrapower induced by a generic $G\subset \mathbb{Q}_\delta$, then letting $\Psi$ be the iteration strategy for the collapsing structure $Q\lhd M$ of $\omega_1^{\P[g]}$, every $A\in \Gamma_\infty^{\P[g][G]}$ is projective in $\Psi$. This means $\Gamma_\infty^{\P[g][G]}\subsetneq j(\Gamma_\infty^{\P[g]})$. Therefore, Tower Sealing fails in $\P[g]$.  
\end{proof}

\begin{remark}
    The proof of the above theorem, particularly Claim \ref{claim:taumainbranch} uses the assumption the set of strong cardinals in the model is finite. The proof can be generalized in a straightforward way to models in which the set of strong cardinals is discrete and has order smaller than the least measurable cardinal. It is not clear how to generalize this proof of hod mice with strong cardinals which reflect the class of strong cardinals. 
\end{remark}

\section{Questions}\label{sec:questions}

\begin{question}
\begin{itemize}
\item Can Tower Sealing hold in a generic extension of a hod mouse?
\item Is Tower Sealing consistent relative to $\sf{ZFC} + $ ``there is a Woodin limit of Woodin cardinals"?
\end{itemize}    
\end{question}
As mentioned in the previous section, it is plausible that some form of Tower Sealing may be shown to hold in hod mice with strong cardinals which reflect the class of strong cardinals; however, the argument has to be different from what is given in this paper. A natural conjecture is

\begin{conjecture}
    Suppose $(\P,\Psi)$ is a hod pair such that $\P\models$ ``there is a strong cardinal which reflects the class of strong cardinals and there is a proper class of Woodin cardinals". Let $\kappa$ be the least strong cardinal which reflects the class of strong cardinals and let $g\subseteq Coll(\omega,\kappa^+)$, then $$\P[g] \models ``\forall \delta \ \textrm{ if } \delta \textrm{ is Woodin, then Tower Sealing holds at } \delta."$$
\end{conjecture}

\medskip

\bibliographystyle{plain}
\bibliography{WeakTS}

@article{steel2022condensation,
  title={Condensation for mouse pairs},
  author={Steel, John and Trang, Nam},
  journal={arXiv preprint arXiv:2207.03559},
  year={2022}
}

@inproceedings{sargsyan2024exact,
  title={The exact consistency strength of the generic absoluteness for the universally Baire sets},
  author={Sargsyan, Grigor and Trang, Nam},
  booktitle={Forum of Mathematics, Sigma},
  volume={12},
  pages={e12},
  year={2024},
  organization={Cambridge University Press}
}

@article{sargsyan2021sealing,
  title={$\sf{Sealing}$ from iterability},
  author={Sargsyan, Grigor and Trang, Nam},
  journal={Transactions of the American Mathematical Society, Series B},
  volume={8},
  number={7},
  pages={229--248},
  year={2021}
}

@article{StrengthuB,
title = {The exact consistency strength of generic absoluteness for universally {B}aire sets},
author = {Sargsyan, Grigor and Trang, Nam},
year = {2019},
note = {In preparation}
}

@article{steel2025JSZ,
title={The JSZ anomaly in strategy mouse comparison},
author = {Steel, John R},
notes={In preparation},
year={2025}
}

@book{normalization_comparison,
author={Steel, John R},
title={A comparison process for mouse pairs},
  SERIES = {Lecture Notes in Logic},
  volume={51},
  year={2022},
  publisher={Cambridge University Press}
}

@incollection {DMATM,
    AUTHOR = {Steel, John R.},
     TITLE = {Derived models associated to mice},
 BOOKTITLE = {Computational prospects of infinity. {P}art {I}. {T}utorials},
    SERIES = {Lect. Notes Ser. Inst. Math. Sci. Natl. Univ. Singap.},
    VOLUME = {14},
     PAGES = {105--193},
 PUBLISHER = {World Sci. Publ., Hackensack, NJ},
      YEAR = {2008},
   MRCLASS = {03E35 (03E55 03E60)},
  MRNUMBER = {MR2449479},
  NOTE={Available at author's website.}
}

@incollection {DMT,
    AUTHOR = {Steel, J. R.},
     TITLE = {The derived model theorem},
 BOOKTITLE = {Logic {C}olloquium 2006},
    SERIES = {Lect. Notes Log.},
    VOLUME = {32},
     PAGES = {280--327},
 PUBLISHER = {Assoc. Symbol. Logic, Chicago, IL},
      YEAR = {2009},
   MRCLASS = {03E35 (03E45 03E55 03E60)},
  MRNUMBER = {2562557},
       DOI = {10.1017/CBO9780511605321.014},
       URL = {https://doi.org/10.1017/CBO9780511605321.014},
}

@book {StationaryTower,
    AUTHOR = {Larson, Paul B.},
     TITLE = {The stationary tower},
    SERIES = {University Lecture Series},
    VOLUME = {32},
      NOTE = {Notes on a course by W. Hugh Woodin},
 PUBLISHER = {American Mathematical Society, Providence, RI},
      YEAR = {2004},
     PAGES = {x+132},
      ISBN = {0-8218-3604-8},
   MRCLASS = {03-02 (03E15 03E35 03E40 03E55 03E60)},
  MRNUMBER = {2069032},
MRREVIEWER = {Miroslav Repick\'{y}},
       DOI = {10.1090/ulect/032},
       URL = {https://doi.org/10.1090/ulect/032},
}

@article{steel2010outline,
  title={An outline of inner model theory},
  author={Steel, John R},
  journal={Handbook of set theory},
  pages={1595--1684},
  year={2010},
  publisher={Springer}
}
\end{document}